\documentclass[notitlepage,reqno,11pt]{amsart}
\usepackage{latexsym,amssymb, epsfig, bm, amsmath,amsfonts, amsthm, mathrsfs}

\usepackage{rotating}
\usepackage[toc,page]{appendix}
\usepackage{color}
\usepackage{multirow}
\usepackage{relsize}
\usepackage{microtype}
\usepackage[foot]{amsaddr}
\usepackage{dsfont}

\usepackage[colorlinks, allcolors=blue]{hyperref}

\usepackage[margin=1in]{geometry}

\numberwithin{equation}{section}
 \newtheorem{assumption}{Assumption}[section]
\newtheorem{lemma}{Lemma}[section]
\newtheorem{theorem}{Theorem}[section]

\newtheorem{coro}{Corollary}[section]
\newtheorem{propo}{Proposition}[section]

\usepackage{xpatch} 
\makeatletter   %
\xpatchcmd{\@thm}{\thm@headpunct{.}}{\thm@headpunct{}}{}{}

\def\E{\mathbb{E}}
\def\P{\mathbb{P}}
\def\R{\mathbb{R}}

\newcommand{\eeq}{\end{equation}}

\newcommand{\eeqa}{\end{eqnarray}}

\newcommand{\baa}{\begin{eqnarray*}}
\newcommand{\eaa}{\end{eqnarray*}}

\newcommand{\ttl}{\Large Spatial SIR epidemic model with varying infectivity  in an unbounded domain: Law of Large Numbers  }

\begin{document}

\title[]{\ttl}

\author[Armand \ Kanga]{Armand Kanga$^1$}
\address{$^1$Laboratory of Applied Mathematics and Computer Science, Universit{\'e} Felix Houphou{\"e}t-Boigny,  Abidjan, Ivory Coast and Aix--Marseille Universit{\'e}, CNRS, I2M, Marseille, France }
\email{kangayaoarmand@gmail.com}
\author[{\'E}tienne \ Pardoux]{{\'E}tienne Pardoux$^2$}
\address{$^2$Aix--Marseille Universit{\'e}, CNRS, I2M, Marseille, France}
\email{etienne.pardoux@univ-amu.fr}

\date{\today}

%
%

\maketitle


\allowdisplaybreaks

%
%

\begin{abstract} 
	We consider a spatial SIR epidemic model where the infectivity of infected individuals depends upon their age of infection, and infections are non local. The domain is an unbounded
	subset of $\R^d$,
	and the individuals do not move. We extend our earlier result in \cite{AK-EP}, where the domain was bounded, and prove a law of large numbers as the size of the population tends to $\infty$.
	\end{abstract}
\section{Introduction}
Epidemic models using ordinary differential equations have been the subject of much research in recent years. Anderson and Britton \cite{HA-TB}, Britton and Pardoux \cite{TB-EP} have shown that these models are limits, when the population size tends towards infinity, of a  stochastic Markovian models. In particular, the Markovian nature of these models implies  that the duration of infection is exponentially distributed, which is unrealistic for most epidemics.\\
As a result, models with non-exponential infection durations have attracted some interest, see  in particular \cite{Rienert} and \cite{Wang}. Kermack and McKendrick \cite{WO-AG} proposed that infectivity should vary with the time since infection. The duration of infection is the time taken by this function to vanish out definitively; its law is completely arbitrary. In \cite{FPP}, the authors have established the law of large numbers for the SIR model with variable infectivity, where the infectivity varies from one individual to another and depends upon the time elapsed since infection. They assume that the infectivity function has a finite number of jumps, and satisfy an assumption of uniform continuity between jumps. In \cite{RF-GP-EP}, the same law of large numbers is established under a much weaker assumption: infectivity functions have their trajectories in $\mathfrak{D}(\mathbb{R},\mathbb{R}_+)$, and are bounded by a constant. However, in the various models studied above, the authors ignore the fact that a population extends over a spatial region. Yet, spatial heterogeneity, habitat connectivity and movement rates play an important role in the evolution of infectious diseases. Both deterministic and stochastic models have been used to understand the importance the geographic dispersion and of the movement of individuals in a population on the spread of infectious diseases, on the persistence or extinction of an endemic disease, for example \cite{LJ-BM}, \cite{Kaj} and \cite{MN-EP-TY}. Some Markovian models in this framework have been studied in \cite{SB-AE-EP}. They studied a stochastic SIR compartmental epidemic model for a population which moves on a torus $(\mathbb{T}^2 = \mathbb{R}^2/\mathbb{Z}^2)$ according to Stochastic Differential Equations driven by independent Brownian motions. They define sequences of empirical measures that describe the evolution of the positions of susceptible, infected and recovered individuals.  They establish large-population approximations of these sequences of  measures. In \cite{YV-MH-EP}, the authors consider a population distributed in the space $\mathbb{R}^d$ whose individuals are characterized by: a position and an infection state, interact and move in $\mathbb{R}^d.$ An epidemic model combining spatial structure and variable infectivity would be more realistic. \\
In a previous work, we studied a no-movement model describing a population distributed within a bounded domain $D\subset \R^d$ \cite{AK-EP} . In this framework, individuals remain fixed at their spatial locations and interact locally according to their epidemiological states (susceptible, infected, or recovered). The assumption that $D$ is bounded simplifies the mathematical analysis and facilitates the derivation of convergence results toward a deterministic limit as the population size tends to infinity.\\
However, this boundedness assumption restricts the applicability of the model to spatially confined configurations and fails to account for systems where the population is distributed over a large or unbounded spatial domain.\\
In the present paper, we extend this framework to the case of an unbounded domain 
$D\subset \R^d$, while maintaining the no-movement assumption.
This extension raises several analytical and probabilistic challenges. We therefore develop refined tools to establish convergence results in this unbounded setting. We introduce sequences of empirical measures describing the spatial distribution and temporal evolution of the susceptible, infected, and recovered individuals.
We then prove a law of large numbers for these measures, characterizing the deterministic limiting system that governs the macroscopic evolution of the population as the number of individuals tends to infinity.
This result extends those obtained in the bounded setting and shows that the model remains well-posed in unbounded spatial domains.
The case where the individuals follow random movements will be considered in another publication
\subsection{Notations}
We note
\begin{itemize}
        \item[•] For $x\in\R^d$, $u>0$, $B(x,u)$ denotes the ball of radius $u$ centered at $x$;
	\item[•] $\mathcal{M}:$ The set of finite measures  on $D$ which we equip with the weak convergence topology;
	\item[•] $\mathfrak{D}:=\mathfrak{D}(\mathbb{R},\mathbb{R}_+):$ The space of càdlàg functions  defined on $\mathbb{R}_+$  with values in  $\mathbb{R}_+;$
	\item[•] $\mathfrak{D}_{\mathcal{M}}:$ The space of càdlàg functions  defined on $\mathbb{R}_+$  with values in $\mathcal{M}.$
	\item[•]$\forall \varphi \in C_b(D)$,  $\forall \mu \in \mathcal{M}, \mu(\varphi)=(\mu,\varphi)= \displaystyle \int_{D} \varphi(x) \mu(dx);$
	\item[•] $c$ and $C$ denote positive constants that can change from line to line.
\end{itemize}
\section{Model description}
\label{k}
The epidemic model studied here is the SIR model in a spatial framework with variable infectivity; the letters S, I and R represent the different states of an individual ("susceptible","infected" and "recovered" respectively). The SIR model states that a  susceptible individual can become infected, and finally  recovered when he/she recovers from the disease and is immune to the disease for ever. In our spatial model, an individual is characterized by its state $E\in \left\{ S, I, R \right\} $ and its position $X$, a continuous variable with values in $D$ which is a closed subset of $\R^d$. To simplify the mathematical description, we identify the $S$, $I$ and $R$ states as $0$, $1$ and $2$ respectively. The  space of individuals is therefore $D\times \left\{ 0, 1, 2 \right\} $. We consider a population of fixed size N; and we assume that at time t=0 the population is divided into three subsets: those susceptible , there are $S^N(0)$ of them, those infected, there are $I^N(0)$ of them, and those removered, there are $R^N(0)$ of them i.e $S^N(0)+I^N(0)+R^N(0)= N.$ At each time, the  individuals occup the positions $X^1, X^2, X^3, \cdots, X^N$ i.e the positions do not change over time.\\
 We denote by,  $\mathfrak{S}_N:=\left\{ i=1,\cdots,N; E^i_0=0\right\}$ the set of  initially susceptible individuals\\  $\left(Card(\mathfrak{S}_N)=S^N(0)\right)$; $\mathfrak{I}_N:=\left\{ i=1,\cdots,N; E^i_0=1\right\}$ the set of  initially infected individuals\\ $\left(Card(\mathfrak{I}_N)=I^N(0)\right)$; and $\mathfrak{R}_N:=\left\{ i=1,\cdots,N; E^i_0=2\right\}$ the set of  initially recovered individuals\\ $\left(Card(\mathfrak{R}_N)=R^N(0)\right)$; $(\mathfrak{S}_N,\mathfrak{I}_N,\mathfrak{R}_N)$ forms a partition of the set $\{1, 2, \cdots, N\}$.
Now let us consider $\{\lambda_{-j} , j\geq 1\}$ and $\{\lambda_{j} , j\geq 1\}$ two mutually independent sequences of  i.i.d random elements  of $\mathfrak{D}.$ $\lambda_{-j}(t)$ is the infectivity at time $t$ of the $j-th$  initially infected individual and $\lambda_{j}(t)$ is the infectivity at time $t$ after its infection of the $j-th$ infected individual among the initially susceptible individuals. We assume that there exists a deterministic constant $\lambda^*>0$ such that $0\leq\lambda_{j}(t)\leq \lambda^*$ a.s, for all $j\in \mathbb{Z}^* \text{ and } t\geq 0$, with the convention: $\forall j \in \mathbb{Z}^*$, $\lambda_{j}(t) = 0 \text{ if } t < 0$ and we shall use the notations $\overline{\lambda}^0(t)= \mathbb{E}(\lambda_{-1}(t))$ and $\overline{\lambda}(t)= \mathbb{E}(\lambda_{1}(t))$. It is natural that an infected individual is more likely to infect a close neighbor than a more distant one. While these different transmission behaviors are averaged in homogeneous SIR models, in our model we use an infection rate that depends on the relative positions of the two parties. The infection rate between two positions will be given by the  symmetric function $K$ defined on $D \times D$ with values in $ \R_+$. A susceptible individual $i$ becomes infected (in other words, his/her state changes from 0 to 1) at time t at rate
\begin{equation}
\label{par1eq1}
\begin{aligned}
\frac{1}{N^{1-\gamma}} \left[\displaystyle  \sum_{j\in\mathfrak{I}_N} \frac{K(X^{i},X^{j})}{\left[\displaystyle \sum_{\ell=1}^N K(X^{\ell},X^{j})\right]^{\gamma}} \lambda_{-j}(t)+ \sum_{j\in\mathfrak{S}_N} \frac{K(X^{i},X^{j})}{\left[\displaystyle \sum_{\ell=1}^N K(X^{\ell},X^{j})\right]^{\gamma}} \lambda_{j}(t- \tau^N(j))\right]
\end{aligned}
\end{equation}
where $\tau^N(j)$ is the infection time of the initially susceptible individual $j$ .\\
\begin{itemize}
\item If $\gamma=1$, the rate  becomes 
\begin{align*}
\displaystyle   \sum_{j\in\mathfrak{I}_N} \frac{K(X^{i},X^{j})}{\displaystyle \sum_{\ell=1}^N K(X^{\ell},X^{j})} \lambda_{-j}(t)+ \sum_{j\in\mathfrak{S}_N} \frac{K(X^{i},X^{j})}{\displaystyle \sum_{\ell=1}^N K(X^{\ell},X^{j})} \lambda_{j}(t- \tau^N(j))
\end{align*}
The interactions are normalized by the sum of the intensities of all contributions received by each individual. Thus, the total quantity affecting an individual $i$ corresponds to a local average rate, depending only on its neighbors and their relative influences. In this case, the rate is not directly influenced by  the total population size.
\item $\gamma=0$  the rate  becomes
\begin{align*}
\frac{1}{N} \left[\displaystyle  \sum_{j\in\mathfrak{I}_N} K(X^{i},X^{j}) \lambda_{-j}(t)+ \sum_{j\in\mathfrak{S}_N} K(X^{i},X^{j})\lambda_{j}(t- \tau^N(j))
\right]
\end{align*}
The interactions are not normalized by the sum of contributions. Each individual directly influences the others, and the total effect on an individual $i$ is averaged over the entire population through the factor $1/N$. Thus, the resulting quantity represents a global average rate, reflecting the collective influence of all individuals.
\end{itemize}
Note also that \eqref{par1eq1} can be rewritten as follows:
\begin{align*}
\displaystyle \sum_{j\in\mathfrak{J}_N}\!\!\left(\frac{\displaystyle\sum_{\ell=1}^NK(X^\ell,X^j)}{N}\right)^{1-\gamma}\!\!\!\frac{K(X^i,X^j)}{\displaystyle\sum_{\ell=1}^NK(X^\ell,X^j)}\lambda_{-j}(t)+\sum_{j\in\mathfrak{S}_N}\!\!\left(\frac{\displaystyle\sum_{\ell=1}^NK(X^\ell,X^j)}{N}\right)^{1-\gamma}\!\!\!\frac{K(X^i,X^j)}{\displaystyle \sum_{\ell=1}^NK(X^\ell,X^j)}\lambda_{j}(t-\tau^N(j)).
\end{align*}
This shows that $\gamma<1$ has the effect of increasing the rate of infections in densely populated regions, this effect being more important with smaller $\gamma$. Finally, due to the fact that our domain is unbounded, for technical reasons which will be transparent below, we will need to assume that $\gamma<1$. Therefore, in all what follows, we will assume that $\gamma\in[0,1)$.\\
At time $t\geq 0$, we define the following
measures:
\begin{equation}
\label{equa1}
\left\lbrace
\begin{aligned}
\mu_t^{S,N} &= \sum_{i\in\mathfrak{S}_N} \mathbf{1}_{E_{t}^{i,N}=0}\delta_{X^{i}}\\
&=\sum_{i\in\mathfrak{S}_N} \delta_{X^{i}}- \sum_{i\in\mathfrak{S}_N} \mathbf{1}_{t\geq \tau^N(i)}\delta_{X^{i}}\\
\mu_t^{I,N}&=\sum_{i\in\mathfrak{I}_N} \mathbf{1}_{E_{t}^{i}=1}\delta_{X^{i}}+  \sum_{i\in\mathfrak{S}_N} \mathbf{1}_{E_{t}^{i,N}=1}\delta_{X^{i}}\\
&=\sum_{i\in\mathfrak{I}_N} \delta_{X^{i}}- \sum_{i\in\mathfrak{I}_N} \mathbf{1}_{\eta_{-i}\leq  t}\delta_{X^{i}}+\sum_{i\in\mathfrak{S}_N} \mathbf{1}_{t\leq \tau^N(i)}\delta_{X^{i}}-
\sum_{i\in\mathfrak{S}_N} \mathbf{1}_{ t \geq \tau^N(i)+ \eta_i  }\delta_{X^{i}}\\
\mu_t^{R,N}&=\sum_{i\in\mathfrak{R}_N} \delta_{X^{i}}+\sum_{i\in\mathfrak{I}_N} \mathbf{1}_{E_{t}^{i}=2}\delta_{X^{i}}+  \sum_{i\in\mathfrak{S}_N} \mathbf{1}_{E_{t}^{i,N}=2}\delta_{X^{i}}\\
&=\sum_{i\in\mathfrak{R}_N} \delta_{X^{i}}+\sum_{i\in\mathfrak{I}_N} \mathbf{1}_{\eta_{-i}\leq t}\delta_{X^{i}}+\sum_{i\in\mathfrak{S}_N} \mathbf{1}_{ \tau^N(i)+ \eta_i \leq t }\delta_{X^{i}}\\
\mu_t^{N}&= \mu_t^{S,N}+\mu_t^{I,N}+\mu_t^{I,N}=\sum_{i\in\mathfrak{S}_N} \delta_{X^{i}}+ \sum_{i\in\mathfrak{I}_N}\delta_{X^{i}} + \sum_{i\in\mathfrak{R}_N}\delta_{X^{i}}\\
&:=\mu^{N}\\
\mu_t^{\mathfrak{F},N} &=  \sum_{i\in\mathfrak{I}_N} \lambda_{-i}(t)\delta_{X^{i}} +  \sum_{i\in\mathfrak{S}_N} \lambda_{i}(t- \tau^N(i))\delta_{X^{i}}.
	\end{aligned}
\right.
\end{equation}
\begin{itemize}
          \item $\eta_j := \sup\{t > 0,\;\lambda_{j}(t) > 0\}$ is the random duration of  the infectious period of individual $j$, $j \in  \mathbb{Z}^*$.
	\item $\mu_t^{S,N}$ is the empirical measure of susceptible individuals at time $t$;
	\item $\mu_t^{\mathfrak{F},N}$ is the empirical measure of the total force of infection at time $t$;
	\item $\mu_t^{I,N}$  is the empirical measure of infected individuals at time $t$ ;
	\item $\mu_t^{R,N}$ is the empirical measure of individuals recovered at time $t$ ;
	\item $\mu_t^{N}$ is the empirical measure of the total population, whihc is independent of $t$.
\end{itemize}
We now renormalize our measures as follows: $\overline{\mu}_t^{S,N}:=\dfrac{1}{N}\mu_t^{S,N}$; $\overline{\mu}_t^{I,N}:=\dfrac{1}{N}\mu_t^{I,N}$; $\overline{\mu}_t^{R,N}:=\dfrac{1}{N}\mu_t^{R,N}$; $\overline{\mu}^{N}:=\dfrac{1}{N}\mu^{N}$ and $\overline{\mu}_t^{\mathfrak{F},N}:=\dfrac{1}{N}\mu_t^{\mathfrak{F},N}$.
We can now rewrite \eqref{par1eq1} as
\begin{align*}
\mathbf{1}_{E_{t}^{i,N}=0}\int_{D} \dfrac{K(X^{i},y)}{\left[\displaystyle \int_{D}K(z,y)\overline{\mu}^{N}(dz) \right]^{\gamma}}\overline{\mu}_t^{\mathfrak{F},N}(dy).
\end{align*}
The two sequences of random variables $\{\eta_{-j} , j\geq 1\}$ and $\{\eta_{j} , j\geq 1\}$  are i.i.d and globally independent of each other. $F(t) := \mathbb{P}(\eta_1\leq t) \text{ and }F_0(t) := \mathbb{P}(\eta_{-1}\leq t)$ are the distribution functions of $\eta_j$ for $j\in \mathbb{Z}_+$ and for $j\in \mathbb{Z}_-$, respectively. For $ i \in \mathfrak{S}_N$, consider a counting process $A_i^N(t),$ which takes the value $0$ when individual $i$ is not yet infected at time t, and takes the value $1$ when the latter has been infected by time t. Thus, $\tau^N(i) := \inf\{t > 0,\; A_i^N(t) = 1\}$. We define $A_i^N$ as follows :
\begin{align*}
A_{i}^N(t)&=	 \int_{0}^{t}\int_{0}^{\infty}\mathbf{1}_{A^N_{i}(s^{-}) = 0}\mathbf{1}_{u\leq \overline{\Gamma}^{N}(s^-,X^i)}P^{i}(ds,du),\quad \text{where}\\
\overline{\Gamma}^{N}(t,x)&=\int_{D} \dfrac{K(x,y)}{\left[\displaystyle \int_{D}K(z,y)\overline{\mu}^{N}(dz) \right]^{\gamma}}\overline{\mu}_{t}^{\mathfrak{F},N}(dy).
\end{align*}
$\{P^{i},\;i \geq 1\}$ are standard Poisson random measures on $\mathbb{R}_{+}^2$ which are mutually independent.\\
The next proposition follows readily from our model
\begin{propo}
	For all $\varphi \in C_b(D)$, $\left\{\overline{\mu}_t^{N}, \overline{\mu}_t^{S,N},\overline{\mu}_t^{\mathfrak{F},N},\overline{\mu}_t^{I,N}, \overline{\mu}_t^{R,N}, t\geq 0 \right\}$ satisfies
	\begin{equation}
	\label{equa2}
	\left\lbrace
	\begin{aligned}
	(\overline{\mu}_t^{S,N},\varphi)&= (\overline{\mu}_0^{S,N},\varphi)
	-\frac{1}{N} \sum_{i\in\mathfrak{S}_N}\varphi(X^{i})A^N_i(t)\\
	(\overline{\mu}_t^{\mathfrak{F},N},\varphi) &= \frac{1}{N} \sum_{i\in \mathfrak{I}_N} \lambda_{-i}(t)\varphi({X^{i}}) +\frac{1}{N} \sum_{i\in \mathfrak{S}_N} \lambda_{i}(t- \tau^N(i))\varphi({X^{i}})\\
	(\overline{\mu}_t^{I,N},\varphi)&= (\overline{\mu}_0^{I,N},\varphi)- \frac{1}{N}\sum_{i\in\mathfrak{I}_N}\varphi(X^{-i})\mathbf{1}_{\eta_{-i}\leq t}\\
	&+\frac{1}{N} \sum_{i\in\mathfrak{S}_N}\varphi(X^{i})A^N_i(t)- \frac{1}{N} \sum_{i\in\mathfrak{S}_N}\varphi(X^{i})\int_0^t\mathbf{1}_{\eta_i\leq t-s}dA^N_i(s)\\
	&= \frac{1}{N}\sum_{i\in\mathfrak{I}_N}\varphi(X^{-i})\mathbf{1}_{\eta_{-i}> t}+
	\frac{1}{N} \sum_{i\in\mathfrak{S}_N}\varphi(X^{i})\int_0^t\mathbf{1}_{\eta_i> t-s}dA^N_i(s)\\
	(\overline{\mu}_t^{R,N},\varphi)&= (\overline{\mu}_0^{R,N},\varphi)+ \frac{1}{N}\sum_{i\in\mathfrak{I}_N}\varphi(X^{-i})\mathbf{1}_{\eta_{-i}\leq t}+\frac{1}{N} \sum_{i\in\mathfrak{S}_N}\varphi(X^{i})\int_0^t\mathbf{1}_{\eta_i\leq t-s}dA^N_i(s)\\
	(\overline{\mu}_t^{N},\varphi)&=(\overline{\mu}_t^{S,N},\varphi)+(\overline{\mu}_t^{I,N},\varphi)+(\overline{\mu}_t^{R,N},\varphi)\\
	&=\frac{1}{N}\sum_{i\in\mathfrak{S}_N} \varphi(X^{i})+ \frac{1}{N}\sum_{i\in\mathfrak{I}_N}\varphi(X^{i})+ \frac{1}{N}\sum_{i\in\mathfrak{R}_N}\varphi(X^{i})\\
	&=(\overline{\mu}_0^{S,N},\varphi)+(\overline{\mu}_0^{I,N},\varphi)+(\overline{\mu}_0^{R,N},\varphi):=(\overline{\mu}^{N},\varphi)
\end{aligned}
	\right.
	\end{equation}	
\end{propo}

\section{Law of large numbers of measures}
\label{kk}
In this section, we first state our assumptions, and then describe the limits of the empirical measures defined in section \ref{k} when the population size tends to infinity. We will then study the limit system, and finally prove the convergence result.
\begin{assumption}
	\label{ass1}
In the following
\begin{itemize}
\item $\overline{S}^N(0):=\frac{S^N(0)}{N}\rightarrow\overline{S}(0);$  $\overline{I}^N(0):=\frac{I^N(0)}{N}\rightarrow\overline{I}(0);$ and $\overline{R}^N(0):=\frac{R^N(0)}{N}\rightarrow\overline{R}(0);$
\item $\displaystyle \P(E^i_0=0)=\overline{S}(0),\quad \P(E^i_0=1)=\overline{I}(0), \quad \text{and}\quad \P(E^i_0=2)=\overline{R}(0)$;
\item The pairs $\left\{\left(E^{i}_0, X^i\right), i=1;\cdots;N\right\}$ are i.i.d;
\item For all $i=1;\cdots;N$, $X^i=\begin{cases}
X_S^i\quad \text{if}\quad E^i_0=0,\\
X_I^i\quad \text{if}\quad E^i_0=1,\\
X_R^i\quad \text{if}\quad E^i_0=2;
\end{cases}$
\item $\left(X_S^i, i=1;\cdots;N\right)$ are  i.i.d with the density function $\pi_S$, $\left(X_I^i, i=1;\cdots;N\right)$ are  i.i.d with the density function $\pi_I$ and $\left(X_R^i, i=1;\cdots;N\right)$  are i.i.d with the density function $\pi_R$. Those three collections of r.v are mutually independent
\end{itemize}
\end{assumption}
\begin{lemma}
\label{Lem1}
Under assumption \ref{ass1}, the sequence $(\overline{\mu}^{N})_{N \geq 1}$ converges a.s. in  $\mathcal{M}$ to $\overline{\mu}$ which verifies
\begin{align}
\overline{\mu}(dx)&=\overline{\mu}^S_0(dx)+\overline{\mu}^I_0(dx)+\overline{\mu}^R_0(dx),\nonumber\\
&= \left(\overline{S}(0)\pi_S(x)+\overline{I}(0)\pi_I(x)+\overline{R}(0)\pi_R(x)\right)dx
\end{align}
\end{lemma}
\begin{proof}
Let $\varphi \in C_b(D)$ and $T>0$. For all $t\in[0;T]$,
\begin{align*}
(\overline{\mu}^{N},\varphi)&=\frac{1}{N}\sum_{i\in\mathfrak{S}_N}\varphi(X^{i})+ \frac{1}{N}\sum_{i\in\mathfrak{I}_N}\varphi(X^{i})+ \frac{1}{N}\sum_{i\in\mathfrak{R}_N}\varphi(X^{i})\\
&=\frac{1}{N}\sum_{i=1}^N\mathbf{1}_{E_0^i=0}\varphi(X_S^{i})+ \frac{1}{N}\sum_{i=1}^N\mathbf{1}_{E_0^i=1}\varphi(X_I^{i})+ \frac{1}{N}\sum_{i=1}^N\mathbf{1}_{E_0^i=2}\varphi(X_R^{i})
\end{align*}
By applying the law of large numbers to each term on the right-hand side, we obtain  that $(\overline{\mu}^{N}, )_{N \geq 1}$ converges a.s. in  $\mathcal{M}$ to $\overline{\mu}$, such that
\begin{align*}
(\overline{\mu},\varphi)&=\overline{S}(0)\int_{D}\varphi(x)\pi_I(x)dx+ \overline{I}(0)\int_{D}\varphi(x)\pi_I(x)dx+ \overline{R}(0)\int_{D}\varphi(x)\pi_R(x)dx,\\
&=\int_{D}\varphi(x)\overline{\mu}^S_0(dx)+ \int_{D}\varphi(x)\overline{\mu}^I_0(dx)+ \int_{D}\varphi(x)\overline{\mu}^R_0(dx)
\end{align*}
We deduce that,
\begin{align*}
\overline{\mu}(dx)&=\overline{\mu}^S_0(dx)+\overline{\mu}^I_0(dx)+\overline{\mu}^R_0(dx).\nonumber\\
&= \left(\overline{S}(0)\pi_S(x)+\overline{I}(0)\pi_I(x)+\overline{R}(0)\pi_R(x)\right)dx\\
\end{align*}
$\overline{\mu}$ is a probability measure, with density $\overline{\mu}(x)= \overline{S}(0)\pi_S(x)+\overline{I}(0)\pi_I(x)+\overline{R}(0)\pi_R(x), \forall x\in D$.\\
We define $S(0,x)\overline{\mu}(x):=\overline{S}(0)\pi_S(x)$; $I(0,x)\overline{\mu}(x):=\overline{I}(0)\pi_I(x)$, and $R(0,x)\overline{\mu}(x):=\overline{R}(0)\pi_R(x)$
\end{proof}
Let $y \in \mathbb{R}^d$, $l_y \in \mathbb{R}^d \setminus \{0\}$, $\|l_y\|=1$ and $\alpha \in [0,\pi)$. We define $C(y,l_y,\alpha):= \left\{\, x \in \mathbb{R}^d : 
\frac{<x - y,l_y>}{\|x - y\|} \ge \cos \alpha \,\right\},$
the cone with vertex $y$, axis parallel to the vector \(l_y\), and (half-)opening angle \(\alpha\).
\begin{assumption}
	\label{ass2}
	We assume that :
	\begin{itemize}
		\item $ \exists r, \underline{c}, \underline{R}, C>0$ such that
		\begin{itemize}
		\item[$i)$]$\forall x,y \in D, \quad K(x,y)\leq C$
		\item[$ii)$]$\|x-y\|\leq r\implies K(x,y)\geq \underline{c},$
		\item[$iii)$] $\|x-y\|> \underline{R}\implies K(x,y)=0$,
		\end{itemize} 
	\item $\exists c_0, C_0, \delta , a > 0$,  $c_0e^{-a\|x\|^{\delta}}\leq\overline \mu(x) \leq C_0e^{-a\|x\|^{\delta}}$;
	\item $\exists \alpha, r > 0, \forall y\in D$ $\exists l_y$, $\|l_y\|=1$ such that $C(y, l_y,\alpha)\cap B(y,r)\subseteq D$ 
			\end{itemize}
\end{assumption}
We shall denote $m:=\displaystyle \int_{C(y,l_y,\alpha)\cap B(y,r)}dz$.

We will assume that there exists a sequence $\{M_n,\ n\ge1\}$ of positive numbers such that $M_1>\underline{R}$ and $M_n\rightarrow \infty$, and for all $n\geq 1$, $D_n:=D\cap B(0,M_n)$ satisfies the last property of Assumption \ref{ass2}.

We shall need the following inequality.
\begin{lemma}
\label{rapp}
Let $\delta, x,y>0$, 
\begin{itemize}
\item If $0<\delta \leq 1$: $(x+y)^\delta \leq x^\delta+y^ \delta$
\item If $\delta \geq 1$: for any $\epsilon>0$, $\exists C_{\epsilon,\delta}$ s.t.  $(x+y)^\delta \leq (1+\epsilon) x^{\delta }+ C_{\epsilon, \delta}y^{\delta}$.
\end{itemize}
\end{lemma}
\begin{proof}
For $0<\delta\le1$, $ x \mapsto f(x)=x^\delta$ is concave on $\mathbb{R}_+$, hence the first part of the result.

If $\delta=2$, we have, for any $\theta>0$,
\[(x+y)^2= x^2+ 2 \sqrt{\theta}x  \frac{y}{\sqrt{\theta}} +y^2\leq (1+ \theta )x^2   +\left(1+\frac{1}{\theta}\right)y^2\,.\]
Iterating this result yields for any $k\in\mathbb{N}$
 \begin{align*}
  (x+y)^{2^k}&\leq (1+\theta )^{2^k-1} x^{2^k}+\left(1+\frac{1}{\theta}\right)^{2^k-1} y^{2^k}\\
  &\le (1+\theta )^{2^k} x^{2^k}+\left(1+\frac{1}{\theta}\right)^{2^k} y^{2^k}
  \end{align*}
  Finally for any $\delta>1$, there exists $k\in\mathbb{N}$ such that $2^{k-1}< \delta \leq 2^k$. Raising the last inequality to the power 
  $\delta 2^{-k}\le1$ and exploiting the first part of the result yields 
  \[(x+y)^\delta\le(1+\theta )^{\delta} x^{\delta}+\left(1+\frac{1}{\theta}\right)^{\delta} y^{\delta},\]
  from which the second part of the result follows by choosing $\theta=(1+\epsilon)^{1/\delta}-1$
\end{proof}

We now state the main result of our paper.
\begin{theorem}
	\label{theo1}
	Under assumptions \ref{ass1} and \ref{ass2}, the sequence $(\overline{\mu}^{S,N}, \overline{\mu}^{\mathfrak{F},N},\overline{\mu}^{I,N},\overline{\mu}^{R,N})_{N \geq 1}$ converges in probability in $\mathfrak{D}_{\mathcal{M}}$  to $(\overline{\mu}^{S}, \overline{\mu}^{\mathfrak{F}},\overline{\mu}^{I},\overline{\mu}^{R})$ such that  $\left\{ \overline{\mu}^S_t, \overline{\mu}^{\mathfrak{F}}_t, \overline{\mu}^I_t, \overline{\mu}^R_t,   t\geq 0\right\}$  satisfies, for all $\varphi \in C_b(D)$:
	\begin{equation}
	\label{sy}
	\left\lbrace
	\begin{aligned}
	(\overline{\mu}_t^{S},\varphi)&= (\overline{\mu}_0^{S},\varphi)- \int_0^t\int_{D}\varphi(x)\overline{\Gamma}(s,x)\overline{\mu}_s^{S}(dx)ds,\\
	(\overline{\mu}_t^{\mathfrak{F}},\varphi) &=\overline{\lambda}^0(t)(\overline{\mu}^I_0,\varphi)+\int_0^t\overline{\lambda}(t-s)\int_{D}\varphi(x)\overline{\Gamma}(s,x)\overline{\mu}_s^{S}(dx)ds,\\
	(\overline{\mu}_t^{I},\varphi)&= (\overline{\mu}_0^{I},\varphi)F_0^c(t)+\int_0^tF^c(t-s)\int_{D} \varphi (x)\overline{\Gamma}(s,x)\overline{\mu}_s^S(dx)ds,\\
	(\overline{\mu}_t^{R},\varphi)&=(\overline{\mu}_0^{R},\varphi)+ (\overline{\mu}_0^{I},\varphi)F_0(t)+\int_0^t F(t-s) \int_{D} \varphi (x)\overline{\Gamma}(s,x)\overline{\mu}_s^S(dx)ds,\\
	\overline{\Gamma}(t,x)&=\int_{D} \dfrac{K(x,y)}{\left[\displaystyle\int_{D}K(z,y)\overline{\mu}(dz) \right]^{\gamma}}\overline{\mu}_t^{\mathfrak{F}}(dy).
	\end{aligned}
	\right.
	\end{equation}	
\end{theorem}
We will first establish the next result 
\begin{propo}
	\label{oo}
	The  system \eqref{sy} admits at most one  solution  $\left(\overline{\mu}^S_t,\overline{\mu}_t^{\mathfrak{F}},\overline{\mu}_t^{I},\overline{\mu}_t^{R}, t \geq 0\right)$   which is absolutely continuous with respect to the measure $\overline{\mu}$, with the densities $\left(S(t,.),\mathfrak{F}(t,.),I(t,.) ,R(t,.), t\geq 0\right)$ satisfying for all $x\in D$
	\begin{equation}
	\label{syste}
	\left\lbrace
	\begin{aligned}
	S(t,x)&= S(0,x)- \int_0^t\int_{D} \Lambda(x,y)\mathfrak{F}(s,y)S(s,x)dyds\\
	\mathfrak{F}(t,x) &=\overline{\lambda}^0(t)I(0,x)+\int_0^t\overline{\lambda}(t-s)\int_{D} \Lambda(x,y)\mathfrak{F}(s,y)S(s,x)dyds.\\
	I(t,x)&= I(0,x)F_0^c(t)+\int_0^t F^c(t-s)\int_{D} \Lambda(x,y)\mathfrak{F}(s,y)S(s,x)ds\\
	R(t,x)&= R(0,x)+I(0,x)F_0(t)+\int_0^tF(t-s) \int_{D} \Lambda(x,y)\mathfrak{F}(s,y)S(s,x)dxdyds\\
\Lambda(x,y)&=\dfrac{K(x,y)\overline{\mu}(y)}{\left[\displaystyle\int_{D}K(z,y)\overline{\mu}(dz) \right]^{\gamma}}.
	\end{aligned}
	\right.
	\end{equation}
\end{propo}
Admitting for a moment the first part of Proposition \ref{oo}, we first establish the following a priori estimates.
\begin{lemma}
	\label{lem1}
For any $\varphi \in L^{\infty} (\mathbb{R}^d)$,  $\left \|\displaystyle \int_{D} \Lambda(.,y)\varphi(y)dy\right \|_{\infty}\leq C\|\varphi\|_{\infty}.$
\end{lemma}
\begin{proof} Let $\varphi \in L^{\infty} (\mathbb{R}^d)$.
\begin{align*}
\left\|\displaystyle \int_{D} \Lambda(.,y)\varphi(y)dy\right \|_{\infty}&=\sup_{x\in D}\int_{D} \dfrac{K(x,y)\overline{\mu}(y)}{\left[\displaystyle\int_{D}K(z,y)\overline{\mu}(z)dz\right]^{\gamma}}\varphi(y)dy,\\
&\leq\frac{C_0 C}{(\underline{c}c_0)^\gamma}\sup_{x\in D}\int_{D}  \dfrac{\mathbf{1}_{\|x-y\|\leq \underline{R}}e^{-a\|y\|^{\delta}}}{\left[\displaystyle\int_{D} \mathbf{1}_{\|z-y\|\leq r}e^{-a\|z\|^{\delta}}dz\right]^{\gamma}}dy\|\varphi\|_{\infty},\\
&\leq \frac{C}{e^{-a\gamma C_\epsilon r^{\delta}}}\sup_{x\in D}\int_{D}  \dfrac{\mathbf{1}_{\|x-y\|\leq \underline{R}}e^{-a\|y\|^{\delta}}}{\left[\displaystyle\int_{C(y,l_y,\alpha)\cap B(y,r)}dz\right]^{\gamma}e^{-a\gamma (1+\epsilon)\|y\|^{\delta}}}dy\|\varphi\|_{\infty},\\
&\leq C\sup_{x\in \R^d}\int_{\R^d} \mathbf{1}_{\|x-y\|\leq \underline{R}}e^{-a(1-\gamma (1+\epsilon))\|y\|^{\delta}}dy\|\varphi\|_{\infty},
\end{align*}
where we have used Lemma \ref{rapp}, in the last inequality.
Since $\gamma <1$, we can choose $\epsilon>0$ such that $\gamma(1+\epsilon) \leq 1$, and we obtain
\begin{align*}
\left\|\displaystyle \int_{D} \Lambda(.,y)dy\right \|_{\infty}&\leq C \sup_{x\in \R^d}\int_{\R^d} \mathbf{1}_{\|x-y\|\leq \underline{R}}dy \|\varphi\|_{\infty}\\
&\leq C\|\varphi\|_{\infty}.
\end{align*}
\end{proof}
Let $\Omega$ be defined on $D\times D$,  by $
\Omega(x,y):=\dfrac{K(x,y)\overline{\mu}(x)}{\left[\displaystyle\int_{D}K(z,y)\overline{\mu}(dz) \right]^{\gamma}}.$
\begin{lemma}
	\label{lem2}
	For any $\varphi \in L^{\infty} (\mathbb{R}^d)$, $\left\|\displaystyle \int_{D} \Omega(x,.)\varphi(x)dx\right\|_{\infty}\leq C\|\varphi\|_{\infty}.$
\end{lemma}
\begin{proof}
Let $\varphi \in L^{\infty} (\mathbb{R}^d)$.
\begin{align*}
\left\|\int_{D} \Omega(x,.)\varphi(x)dx\right\|_{\infty}&=\sup_{y\in D} \int_{D}\dfrac{K(x,y)\overline{\mu}(x)}{\left[\displaystyle\int_{D}K(z,y)\overline{\mu}(dz) \right]^{\gamma}}\varphi(x)dx,\\
&\leq \frac{C}{e^{-a\gamma C_\epsilon r^{\delta}}}\sup_{y\in D}\int_{D}  \dfrac{\mathbf{1}_{\|x-y\|\leq \underline{R}}e^{-a\|x\|^{\delta}}}{\left[\displaystyle\int_{C(y,l_y,\alpha)\cap B(y,r)}dz\right]^{\gamma}e^{-a\gamma (1+\epsilon)\|y\|^{\delta}}}dx\|\varphi\|_{\infty},\\
\left\|\int_{D} \Omega(x,.)\varphi(x)dx\right\|_{\infty}&\leq C\sup_{y\in \R^d}\int_{\R^d}  \mathbf{1}_{\|x-y\|\leq \underline{R}}e^{-a\|x\|^{\delta}}e^{a\gamma (1+\epsilon)\|y\|^{\delta}}dx\|\varphi\|_{\infty}.
\end{align*}
Morover, $\|y\|^{\delta} \leq  (1+\epsilon) \|x\|^{\delta} + C_{\epsilon}\|x-y\|^{\delta}$. Then
\begin{align*}
\left\|\int_{D} \Omega(x,.)\varphi(x)dx\right\|_{\infty}&\leq C\sup_{y\in \R^d}\int_{\R^d}  \mathbf{1}_{\|x-y\|\leq \underline{R}}e^{aC_{\epsilon}\|x-y\|^{\delta}}e^{-a\left(1-\gamma (1+\epsilon)^2\right)\|y\|^{\delta}} dx\|\varphi\|_{\infty}\\
&\leq Ce^{aC_{\epsilon}\underline{R}^{\delta}}\sup_{y\in \R^d}  \int_{\R^d} \mathbf{1}_{\|x-y\|\leq \underline{R}}dx\|\varphi\|_{\infty},
\end{align*}
where we have choosen $\epsilon>0$ such that $\gamma(1+\epsilon)^2\leq 1$. The results follows.
\end{proof}
\begin{lemma}
	\label{lem3}
	Let $T>0$, and let $(S,\mathfrak{F} )$  be a solution of the first two equations of \eqref{syste}. Then there exists a positive constant $C$ such that:
	\begin{itemize}
		\item $\forall t\in[0;T],$ $\|S(t)\|_{\infty}\leq C;$
		\item $\forall t\in[0;T]$; $\|\mathfrak{F}(t)\|_{\infty}\leq C.$
	\end{itemize}
\end{lemma}
\begin{proof}
Let $t\in [0;T]$ and $\forall x \in D$.
\begin{align*}
S(t,x)&\leq S(0,x)\nonumber\\
\|S(t)\|_{\infty}&\leq\|S(0)\|_{\infty}\nonumber\\
&\leq C.
\end{align*}
\begin{align*}
\mathfrak{F}(t,x) &=\overline{\lambda}^0(t)I(0,x)+\int_0^t\overline{\lambda}(t-s)S(s,x)\int_{D} \Lambda(x,y)\mathfrak{F}(s,y)dyds\\
\|\mathfrak{F}(t)\|_{\infty} &\leq\lambda^*\|I(0)\|_{\infty}+\lambda^*\int_0^t \|S(s)\|_{\infty}\left\|\int_{D} \Lambda(.,y)\mathfrak{F}(s,y)dy\right\|_{\infty}ds\\
&\leq \lambda^*\|I(0)\|_{\infty}+\lambda^*C\int_0^t \left\|\int_{D} \Lambda(.,y)\mathfrak{F}(s,y)dy\right\|_{\infty}ds
\end{align*}
Combining with the result of Lemma \ref{lem1}, we obtain
\begin{align*}
\|\mathfrak{F}(t)\|_{\infty} &\leq\lambda^*\|I(0)\|_{\infty}+\lambda^*C\int_0^t \|\mathfrak{F}(s)\|_{\infty}ds
\end{align*}
Applying Gronwall's inequality, we deduce that
\begin{align*}
\|\mathfrak{F}(t)\|_{\infty} &\leq C,\quad  C:=\lambda^*\|I(0)\|_{\infty}e^{\lambda^*CT}, \quad t\leq T
\end{align*}
\end{proof}
\begin{proof}\textit{of Proposition \ref{oo}}.
Step 1: We first show that for all $t\geq 0$ any solution of \eqref{sy}, $\left(\overline{\mu}^S_t,\overline{\mu}_t^{\mathfrak{F}},\overline{\mu}_t^{I},\overline{\mu}_t^{R}\right)$ is absolutely continuous with respect to the measure $\overline{\mu}$, and the densities\\ $\left(S(t,.), \mathfrak{F}(t,.), I(t,.) ,R(t,.)\right)$ verify \eqref{syste}.
 From the first equation of \eqref{sy}, $\overline{\mu}^S_t\leq \overline{\mu}^S_0$. Since $\overline{\mu}^S_0$ is absolutely continuous,  $\overline{\mu}^S_t$ has the same property, and we denote its density by $\overline{\mu}^S(t,x)$. By setting $\overline{\mu}^S(t,x)= S(t,x)\overline{\mu}(x)$, thus $S(t,)$ is the density of $\overline{\mu}^S_t$ with respect to the measure $\overline{\mu}$.\\ 
From the third equation of \eqref{sy}, $\displaystyle \overline{\mu}_t^{I}\leq  \overline{\mu}_0^{I}+ \int_0^t \overline{\Gamma}(s,.)\overline{\mu}^{S}_sds$, thus $\overline{\mu}_t^{I}$ is absolutely continuous , since $\overline{\mu}^I_0$ is absolutely continuous, as well as $\overline{\mu}^{S}_s$ for all s. 
A similar argument applies  to $ \overline{\mu}_t^{\mathfrak{F}}$ and $ \overline{\mu}_t^{R}$.The system of equation \eqref{syste} now follows readily from \eqref{sy}.\\
\smallskip

Step 2: We will verify that  equation \eqref{syste} has at most one solution, which implies that \eqref{sy} has at most one solution.  For that sake, it suffices to show that the system made of the first two equations of \eqref{syste}  has most one solution. The first two equations of the system \eqref{syste} constitute the following system
	\begin{equation}
	\label{sys}
	\left\lbrace
	\begin{aligned}
		S(t,x)&= S(0,x)- \int_0^t\int_{D} \Lambda(x,y)\mathfrak{F}(s,y)S(s,x)dyds,\\
	\mathfrak{F}(t,x) &=\overline{\lambda}^0(t)I(0,x)+\int_0^t\overline{\lambda}(t-s)\int_{D} \Lambda(x,y)\mathfrak{F}(s,y)S(s,x)dyds,\\
 \Lambda(x,y)&=\dfrac{K(x,y)\overline{\mu}(y)}{\left[\displaystyle\int_{D}K(z,y)\overline{\mu}(dz) \right]^{\gamma}}\,.
	\end{aligned}
	\right.
	\end{equation}
	Let $\left(f_1(t,.), g_1(t,.)\right)$ and $\left(f_2(t,.), g_2(t,.)\right)$ be two solutions of the above system with the same initial condition.\\
	On the one hand	
	\begin{align}
	f_1(t,x)-f_2(t,x)&= \int_0^t  \left(f_2(s,x)-f_1(s,x)\right)\int_{D}\Lambda(x,y)g_2(s,y)dyds\nonumber\\
	&+\int_0^tf_1(s,x)\int_{D}\Lambda(x,y)\left( g_2(s,y)-g_1(s,y)\right)dyds\nonumber
	\end{align}
Using the result of Lemma \ref{lem1}, we obtain
	\begin{align}
	\|f_1(t,.)-f_2(t,.)\|_{\infty}&\leq C\int_0^t \|f_2(s,.)-f_1(s,.)\|_{\infty}\|g_2(s,.)\|_{\infty}ds\nonumber\\
	&+ C\int_0^t\| g_2(s,.)-g_1(s,.)\|_{\infty} \|f_1(s,.)\|_{\infty}ds\nonumber
	\end{align}
This combined with Lemma \ref{lem3}, implies that
	\begin{align}
	\label{equa5}
	\|f_1(t,.)-f_2(t,.)\|_{\infty}&\leq C\int_0^t \left(\|f_2(s,.)-f_1(s,.)\|_{\infty}+\| g_2(s,.)-g_1(s,.)\|_{\infty}\right) ds
	\end{align}
	Moreover,
	\begin{align}
	g_1(t,x)-g_2(t,x)&=\int_0^t\overline{\lambda}(t-s) (f_1(s,x)-f_2(s,x))\int_{D}\Lambda(x,y)g_1(s,y)dyds\nonumber\\
	&+\int_0^t \overline{\lambda}(t-s)f_1(s,x)\int_{D}\Lambda(x,y) (g_1(s,y)-g_2(s,y))dyds\nonumber\\
	|g_1(t,x)-g_2(t,x)|&\leq C\int_0^t \| f_1(s,.)-f_2(s,.)\|_{\infty}\left\|\int_{D}\Lambda(x,y)g_1(s,y)dy\right\|_{\infty}ds\nonumber\\
	&+C \int_0^t \left\|\int_{D}\Lambda(x,y) (g_1(s,y)-g_2(s,y))dy\right\|_{\infty}ds\nonumber
	\end{align}
	Using the results of Lemma \ref{lem1} and \ref{lem3}, we obtain
	\begin{align}
	\label{equa6}
\|g_1(t,.)-g_2(t,.)\|_{\infty}	&\leq C\int_0^t \left(\| f_1(s,.)-f_2(s,.)\|_{\infty}+\|g_1(s,.)-g_2(s,.)\|_{\infty}\right)ds
	\end{align}
	From \eqref{equa5} and \eqref{equa6}, we infer that
	\begin{align*}
	\|g_1(t,.)-g_2(t,.)\|_{\infty}+\|f_1(t,.)-f_2(t,.)\|_{\infty}&\leq C \int_0^t  \left(\|g_1(s,.)-g_2(s,.)\|_{\infty}+\|f_1(s,.)-f_2(s,.)\|_{\infty}\right)ds
	\end{align*}
	Using Gronwall's inequality, we conclude that 
	\begin{align*}
	\|g_1(t,.)-g_2(t,.)\|_{\infty}+\|f_1(t,.)-f_2(t,.)\|_{\infty}&=0. 
	\end{align*}
\end{proof}

$M_n$ and $D_n$ being defined as below Assumption  \ref{ass2}, we study the dynamics of the disease in this domain . A susceptible individual $i$ located in $D_{n}$, becomes infected  at time $t$ at rate
\begin{equation}
\label{2}
\begin{aligned}
\overline{\Gamma}^N_{n}(t,X^i)&=\frac{1}{N}\displaystyle  \sum_{j\in\mathfrak{I}_N} \frac{K(X^{i},X^{j})}{\left[\displaystyle \frac{1}{N} \sum_{\ell=1}^N K(X^{\ell},X^{j})\mathbf{1}_{\|X^{\ell}\|\leq M_n}\right]^{\gamma}} \lambda_{-j}(t) \mathbf{1}_{\|X^j\|\leq M_n} \\
&+ \frac{1}{N}\sum_{j\in\mathfrak{S}_N} \frac{K(X^{i},X^{j})}{\left[\displaystyle \frac{1}{N} \sum_{\ell=1}^N K(X^{\ell},X^{j})\mathbf{1}_{\|X^{\ell}\|\leq M_n}\right]^{\gamma}} \lambda_{j}(t- \tau_{n}^N(j))\mathbf{1}_{\|X^j\|\leq M_n}.
\end{aligned}
\end{equation}
where $\tau_{n}^N(j)$ is the infection time of the initially susceptible individual $j$ . Like in section \ref{k}, at time $t\geq 0$ we define the following measures:
\begin{equation}
\left\lbrace
\begin{aligned}
\overline{\mu}_{n,t}^{S,N}&=\frac{1}{N}\sum_{i\in\mathfrak{S}_N} \delta_{X^{i}}\mathbf{1}_{\|X^i\|\leq M_n}-\frac{1}{N} \sum_{i\in\mathfrak{S}_N} A^N_{n,i}(t)\delta_{X^{i}}\mathbf{1}_{\|X^i\|\leq M_n}\\\
\overline{\mu}_{n,t}^{I,N}&=\frac{1}{N}\sum_{i\in\mathfrak{I}_N} \delta_{X^{i}}\mathbf{1}_{\|X^i\|\leq M_n}- \frac{1}{N}\sum_{i\in\mathfrak{I}_N} \mathbf{1}_{\eta_{-i}\leq  t}\delta_{X^{i}}\mathbf{1}_{\|X^i\|\leq M_n}\\
&+\frac{1}{N}\sum_{i\in\mathfrak{S}_N}A^N_{n,i}(t)\delta_{X^{i}}\mathbf{1}_{\|X^i\|\leq M_n}-\frac{1}{N}
\sum_{i\in\mathfrak{S}_N}\int_0^t\mathbf{1}_{\eta_{i}\leq t-s}dA^N_{n,i}(s)\delta_{X^{i}}\mathbf{1}_{\|X^i\|\leq M_n}\\
\overline{\mu}_{n,t}^{R,N}&=\frac{1}{N}\sum_{i\in\mathfrak{R}_N} \delta_{X^{i}}\mathbf{1}_{\|X^i\|\leq M_n}+\frac{1}{N}\sum_{i\in\mathfrak{I}_N} \mathbf{1}_{\eta_{-i}\leq t}\delta_{X^{i}}\mathbf{1}_{\|X^i\|\leq M_n}\\
&+\frac{1}{N}\sum_{i\in\mathfrak{S}_N} \int_0^t\mathbf{1}_{\eta_{i}\leq t-s}dA^N_{n,i}(s)\delta_{X^{i}}\mathbf{1}_{\|X^i\|\leq M_n}\\\
\overline{\mu}_{n,t}^{N}
&=\frac{1}{N}\sum_{i\in\mathfrak{S}_N} \delta_{X^{i}}\mathbf{1}_{\|X^i\|\leq M_n}+ \frac{1}{N}\sum_{i\in\mathfrak{I}_N} \delta_{X^{i}}\mathbf{1}_{\|X^i\|\leq M_n}+ \frac{1}{N}\sum_{i\in\mathfrak{R}_N} \delta_{X^{i}}\mathbf{1}_{\|X^i\|\leq M_n}\\
\overline{\mu}_{n,t}^{\mathfrak{F},N} &=  \frac{1}{N}\sum_{i\in\mathfrak{I}_N} \lambda_{-i}(t)\delta_{X^{i}}\mathbf{1}_{\|X^i\|\leq M_n} +  \frac{1}{N}\sum_{i\in\mathfrak{S}_N} \int_0^t \lambda_i(t-s)dA^N_{n,i}(s)\delta_{X^{i}}\mathbf{1}_{\|X^i\|\leq M_n},
\end{aligned}
\right.
\end{equation}
 with $\tau_{n}^N(i):= \inf\{t > 0,\;A_{n,i}^N(t) = 1\}$, where
\begin{align*}
A_{n,i}^N(t)&=	 \int_{0}^{t}\int_{0}^{\infty}\mathbf{1}_{A^N_{n,i}(s^{-}) = 0}\mathbf{1}_{u\leq \overline{\Gamma}_{n}^{N}(s^-,X^i)}Q^{i}(ds,du),\quad \text{where}\\
\overline{\Gamma}_{n}^{N}(t,x)&=\int_{D} \dfrac{K(x,y)}{\left[\displaystyle \int_{D_{n}}K(z,y)\overline{\mu}^{N}(dz) \right]^{\gamma}}\overline{\mu}_{n,t}^{\mathfrak{F},N}(dy).
\end{align*}
$\{Q^{i},\;i \geq 1\}$ are standard Poisson random measures on $\mathbb{R}_{+}^2$ which are mutually independent and do not depend upon n. $D_{n}$ is a bounded set, thus from the above work (see Theorem 3.1 of \cite{AK-EP}), we deduce that $\left(\overline{\mu}_{n}^{S,N},\overline{\mu}_{n}^{\mathfrak{F},N},\overline{\mu}_{n}^{I,N},\overline{\mu}_{n}^{R,N} \right)$ converges, as $N \rightarrow \infty$ in probability to $(\overline{\mu}_{n}^{S}, \overline{\mu}_{n}^{\mathfrak{F}},\overline{\mu}_{n}^{I},\overline{\mu}_{n}^{R})$ such that for all $\varphi \in C_b(D)$, $ t\geq 0,$ 
\begin{equation}
\label{gg}
\left\lbrace
\begin{aligned}
(\overline{\mu}_{n,t}^{S},\varphi)&= (\overline{\mu}_{n,0}^{S},\varphi)- \int_0^t\int_{D}\varphi(x)\overline{\Gamma}_{n}(s,x)\overline{\mu}_{n,s}^{S}(dx)ds,\\
(\overline{\mu}_{n,t}^{\mathfrak{F}},\varphi) &=\overline{\lambda}^0(t)(\overline{\mu}^I_{n,0},\varphi)+\int_0^t\overline{\lambda}(t-s)\int_{D}\varphi(x)\overline{\Gamma}_{n}(s,x)\overline{\mu}_{n,s}^{S}(dx)ds,\\
(\overline{\mu}_{n,t}^{I},\varphi)&= (\overline{\mu}_{n,0}^{I},\varphi)F_0^c(t)+\int_0^tF^c(t-s)\int_{D} \varphi (x)\overline{\Gamma}_{n}(s,x)\overline{\mu}_{n,s}^S(dx)ds,\\
(\overline{\mu}_{n,t}^{R},\varphi)&=(\overline{\mu}_{n,0}^{R},\varphi)+ (\overline{\mu}_{n,0}^{I},\varphi)F_0(t)+\int_0^t F(t-s) \int_{D} \varphi (x)\overline{\Gamma}_{n}(s,x)\overline{\mu}_{n,s}^S(dx)ds,\\
\overline{\Gamma}_{n}(t,x)&=\int_{D} \dfrac{K(x,y)}{\left[\displaystyle\int_{D_{n}}K(z,y)\overline{\mu}(dz) \right]^{\gamma}}\overline{\mu}_{n,t}^{\mathfrak{F}}(dy).
\end{aligned}
\right.
\end{equation}
A proof similar to that of Proposition 3.1  \label{AK-EP} shows that, the system \eqref{gg} admits at most one solution  $\left\{\overline{\mu}^S_{n,t},\overline{\mu}_{n,t}^{\mathfrak{F}},\overline{\mu}_{n,t}^{I},\overline{\mu}_{n,t}^{R}, t \geq 0\right\}$   which is absolutely continuous with respect to the measure $\overline{\mu}$, with the densities $\left\{ S_{n}(t,.),\mathfrak{F}_{n}(t,.),I_{n}(t,.) ,R_{n}(t,.), t\geq 0\right\}$ satisfying for all $x\in D$
\begin{equation}
\label{syste1}
\left\lbrace
\begin{aligned}
S_{n}(t,x)&= S_{n}(0,x)- \int_0^t\int_{D} \Lambda_{n}(x,y)\mathfrak{F}_{n}(t,y)S_{n}(s,x)dyds\\
\mathfrak{F}_{n}(t,x) &=\overline{\lambda}^0(t)I_{n}(0,x)+\int_0^t\overline{\lambda}(t-s)\int_{D} \Lambda_{n}(x,y)\mathfrak{F}_{n}(s,y)S_{n}(s,x)dyds.\\
I_{n}(t,x)&= I_{n}(0,x)F_0^c(t)+\int_0^t F^c(t-s)\int_{D} \Lambda_{n}(x,y)\mathfrak{F}_{n}(s,y)S_{n}(s,x)dyds\\
R_{n}(t,x)&= R_{n}(0,x)+I_{n}(0,x)F_0(t)\\
&+\int_0^tF(t-s) \int_{D} \Lambda_{n}(x,y)\mathfrak{F}_{n}(s,y)S_{n}(s,x)dxdyds\\
\Lambda_{n}(x,y)&=\dfrac{K(x,y)\overline{\mu}(y)}{\left[\displaystyle\int_{D_{n}}K(z,y)\overline{\mu}(dz) \right]^{\gamma}}.\end{aligned}
\right.
\end{equation}

\bigskip

We shall use below the following construction. Consider the partition of $\R^d$ made of translates of the hypercube $(0,\frac{a}{\sqrt{d}}]^d$, where the value of $a$ will be specified in the proof of the next Lemma.
Let $(\mathfrak{A}_1, \mathfrak{A}_2,\ldots)$ denote the subset of the elements of that partition which are contained in $D$. We assume that the numbering is chosen in such a way that the sequence of the distances from the center of $\mathfrak{A}_k$ to $0$ is a non decreasing function of $k$. 

We next define the mapping $q:[M_1\vee a,+\infty)\mapsto \mathbb{N}$  which to each $u\ge M_1\vee a$ associates 
\[ q(u)=\inf\{k\ge1,\ \mathfrak{A}_{k+1}\cap B(0,u)^c\not=\emptyset\}\,.\]
In other words, $(\mathfrak{A}_1,\ldots,\mathfrak{A}_{q(u)})$ is the list of all elements of our partition which are included in $D\cap B(0,u)$.
\begin{lemma}
\label{impor}
For all $n\ge 1$,  $y\in D_{n}$, there exists $  \alpha ,r>0$ and $1\leq k \leq q(M_n)$ such that $\mathfrak{A}_k \subset C(y,l_y, \alpha)\cap B(y,r)$.
\end{lemma}
\begin{proof}
Let $n\ge1$ and $y\in D_{n}$. According to the hypothesis, $\exists \alpha,r>0$ such that $C(y,l_y, \alpha)\cap B(y,r) \subset D_{n}$.
 Let now $u(y):=y+\frac{r}{1+\sin(\alpha)}l_y$ and $a:=\frac{r\sin(\alpha)}{1+\sin(\alpha)}.$ Then,  $B(u(y),a)\subset C(y,l_y, \alpha)\cap B(y,r)  $. In fact,  $\forall z\in B(u(y),a)$ we have
 \begin{itemize}
 \item $ \|y-z\|\leq \|y-u(y)\|+\|u(y)-z\| \leq \frac{r}{1+\sin(\alpha)}+\frac{r\sin(\alpha)}{1+\sin(\alpha)}=r\implies z \in B(y,r)$ .
\item The minimum distance from $u(y)$ to the boundary of the cone $C(y,l_y,\alpha)$ is\\ $\|u(y)-y\|\sin (\alpha)= \frac{r\sin (\alpha)}{1+\sin(\alpha)}\geq \|z-u(y)\|\implies z \in C(y,l_y, \alpha)$.
\end{itemize}
Let $(e_1,\cdots,e_d)$ denote the orthonormal basis of $\R^d$ which has been implicitly used in the definition of the sequence $(\mathfrak{A}_k,\ k\ge1)$. The hypercube $H:=\left\{ u(y)+\displaystyle \sum_{i=1}^d t_i e_i, |t_i|\leq \frac{a}{\sqrt{d}}\right\}$ is contained in $B(u(y),a)$. 
It is plain that there exists $1\le k\le q(M_n)$ such that $\mathfrak{A}_k\subset H$: we can choose the unique $k$ s.t. $\mathfrak{A}_k$  contains $u(y)$.
\end{proof}
The same argument as that used in the proof of Lemma \ref{lem3} yields the following result
\begin{lemma}
	\label{lem4}
	Let $T>0$, and let $(S_{n},\mathfrak{F}_{n} , n\in \mathbb{N})$  be a solution of the first two equations of \eqref{syste1}. Then there exists a positive constant $C$ such that: $\forall n \in \mathbb{N}^*$,
	\begin{itemize}
		\item $\forall t\in[0;T],$ $\|S_{n}(t)\|_{\infty}\leq C;$
		\item $\forall t\in[0;T]$; $\|\mathfrak{F}_{n}(t)\|_{\infty}\leq C.$
	\end{itemize}
\end{lemma}
We now prove that the deterministic model in $D$ is well approximate by the deterministic model in $D_n$
\begin{theorem}
\label{Thé1}
	$(\overline{\mu}_{n}^{S}, \overline{\mu}_{n}^{\mathfrak{F}},\overline{\mu}_{n}^{I},\overline{\mu}_{n}^{R})$ converges in  $\mathfrak{D}^4_{\mathcal{M}}$ to $(\overline{\mu}^{S}, \overline{\mu}^{\mathfrak{F}},\overline{\mu}^{I},\overline{\mu}^{R})$, as $n \rightarrow \infty$.
\end{theorem}
\begin{proof}
To prove this result, we prove that   $\left\{\left(S_{n}(t,.), \mathfrak{F}_{n}((t,.), I_{n}((t,.), R_{n}((t,.)\right),  t\geq 0\right\}$ converges in $L^1 (\overline{\mu})$ locally uniformly in t to $\displaystyle \left\{\left(S(t,.), \mathfrak{F}(t,.), I(t,.), R(t,.)\right),  t\geq 0\right\}$, as $n\rightarrow \infty$. We will prove that, locally uniformly in t ,    $(S_{n}(t,.), \mathfrak{F}_{n}(t,.))$ converges in $L^1 (\overline{\mu})$ to $(S(t,.), \mathfrak{F}(t,.))$ as $n\rightarrow \infty$. The rest of the result then follows easily.\\
Let $T>0$.  For all $t\in[0;T]$, we have, exploiting Lemma \ref{lem1}, \ref{lem2} and \ref{lem3}
\begin{align}
	\left|S(t,x)-S_{n}(t,x)\right|&\leq \left|S(0,x)- S_{n}(0,x)\right|\nonumber\\
	&+\left|\int_0^t\int_{D}\left( \Lambda_{n}(x,y)\mathfrak{F}_{n}(s,y)S_{M_n}(s,x)dyds- \Lambda(x,y)\mathfrak{F}(s,y)S(s,x)\right)dyds\right|\nonumber,\\
	&\leq \left\|S(0)-S_{n}0)\right\|_{L^1 (\overline{\mu})}+\int_0^t\left\| (S(s)-S_{n}s))\int_{D}\Lambda(.,y)\mathfrak{F}(s,y)dy\right\|_{L^1 (\overline{\mu})}ds\nonumber\\
&+\int_0^t \left\|S_{n}(s)\int_{D} \Lambda(.,y)(\mathfrak{F}(s,y)-\mathfrak{F}_{n}(s,y))dy\right\|_{L^1 (\overline{\mu})}ds\nonumber\\
&+ \int_0^t\left\| S_{n}(s)\int_{D}\mathfrak{F}_{n}(s,y)\left(\Lambda(.,y)-\Lambda_{n}(.,y)\right)dy\right\|_{L^1 (\overline{\mu})}ds,\nonumber\\
\left\|S(t)-S_{n}(t)\right\|_{L^1 (\overline{\mu})}&\leq \left\|S(0)-S_{n}(0)\right\|_{L^1 (\overline{\mu})}+\int_0^t\left\| (S(s)-S_{n}(s))\right\|_{L^1 (\overline{\mu})}\left\|\int_{D}\Lambda(.,y)\mathfrak{F}(s,y)dy\right\|_{\infty}ds\nonumber\\
&+\int_0^t \left\|S_{n}(s)\right\|_{\infty}\left\|\int_{D} \Lambda(.,y)(\mathfrak{F}(s,y)-\mathfrak{F}_{n}(s,y))dy\right\|_{L^1 (\overline{\mu})}ds\nonumber\\
&+ \int_0^t\left\| S_n(s)\right\|_{\infty}\left\|\int_{D_{n}}\left(\Lambda(.,y)-\Lambda_{n}(.,y)\right)\mathfrak{F}_{n}(s,y)dy\right\|_{L^1 (\overline{\mu})}ds,\nonumber\\
&\leq \left\|S(0)-S_{n}(0)\right\|_{L^1 (\overline{\mu})}+C\int_0^t\left\| (S(s)-S_{n}(s))\right\|_{L^1 (\overline{\mu})}\left\|\mathfrak{F}(s)\right\|_{\infty}ds\nonumber\\
&+\int_0^t \left\|S_{n}(s)\right\|_{\infty}\left\|\mathfrak{F}(s)-\mathfrak{F}_{n}(s)\right\|_{L^1 (\overline{\mu})}\left\|\int_{D} \Omega(x,.) dx\right\|_{\infty}ds\nonumber\\
&+ \int_0^t\left\| S_{n}(s)\right\|_{\infty}\left\| \mathfrak{F}_{n}(s)\right\|_{\infty}\left\|\int_{D}\left(\Lambda(.,y)-\Lambda_{n}(.,y)\right)\mathbf{1}_{\|y\|\leq M_n}dy\right\|_{L^1 (\overline{\mu})}ds,\nonumber\\
&\leq \left\|S(0)-S_{n}(0)\right\|_{L^1 (\overline{\mu})}\nonumber \\
&+C\int_0^t\left(\left\| (S(s)-S_{n}(s))\right\|_{L^1 (\overline{\mu})}+\left\|\mathfrak{F}(s)-\mathfrak{F}_{n}(s)\right\|_{L^1 (\overline{\mu})}\right)ds\nonumber\\
\label{ineq1}
&+CT\left\|\int_{D}\left(\Lambda(.,y)-\Lambda_{n}(.,y)\right)\mathbf{1}_{\|y\|\leq M_n}dy\right\|_{L^1 (\overline{\mu})},
\end{align}
\begin{align*}
\left|\mathfrak{F}(t,x)-\mathfrak{F}_{n}(t,x)\right|&\leq \lambda^*\left|I(0,x)- I_{n}(0,x)\right|\\
&+\lambda^*\left|\int_0^t\int_{D}\left( \Lambda_{n}(x,y)\mathfrak{F}_{n}(s,y)S_{n}(s,x)dyds- \Lambda(x,y)\mathfrak{F}(s,y)S(s,x)\right)dyds\right|.
\end{align*}
Using the reasoning leading to \eqref{ineq1}, we deduce that
\begin{align}
\left\|\mathfrak{F}(t)-\mathfrak{F}_{n}(t)\right\|_{L^1 (\overline{\mu})}&\leq \lambda^*\left\|I(0)- I_{M_n}(0)\right\|_{L^1 (\overline{\mu})}\nonumber\\
&+C\int_0^t\left(\left\| S(s)-S_{n}(s)\right\|_{L^1 (\overline{\mu})}+\left\|\mathfrak{F}(s)-\mathfrak{F}_{n}(s)\right\|_{L^1 (\overline{\mu})}\right)ds\nonumber\\
\label{ineq2}
&+CT\left\|\int_{D}\left(\Lambda(.,y)-\Lambda_{n}(.,y)\right)\mathbf{1}_{\|y\|\leq M_n}dy\right\|_{L^1 (\overline{\mu})}.
\end{align}
From \eqref{ineq1} and \eqref{ineq2}, we have
\begin{align}
\left\|S(t)-S_{n}(t)\right\|_{L^1 (\overline{\mu})}+\left\|\mathfrak{F}(t)-\mathfrak{F}_{n}(t)\right\|_{L^1 (\overline{\mu})}&\leq \lambda^*\left\|I(0)- I_{n}(0)\right\|_{L^1 (\overline{\mu})}+\left\|S(0)- S_{n}(0)\right\|_{L^1 (\overline{\mu})}\nonumber\\
\label{ineq3}
&+C\int_0^t\left\| S(s)-S_{n}(s)\right\|_{L^1 (\overline{\mu})}+\left\|\mathfrak{F}(s)-\mathfrak{F}_{n}(s)\right\|_{L^1 (\overline{\mu})}ds\\
&+C\left\|\int_{D}\left(\Lambda(.,y)-\Lambda_{n}(.,y)\right)\mathbf{1}_{\|y\|\leq _{M_n}}dy\right\|_{L^1 (\overline{\mu})}\nonumber.
\end{align}
Applying Gronwall's inequality to \eqref{ineq3},  since $\left\|I(0)- I_{n}(0)\right\|_{L^1 (\overline{\mu})} \xrightarrow[n \to \infty]{} 0$, as well as \\$\left\|S(0)- S_{n}(0)\right\|_{L^1 (\overline{\mu})} \xrightarrow[n \to \infty]{} 0$, it remains to show that   $$\displaystyle \Pi_{n}=\left\|\int_{D}\left(\Lambda(.,y)-\Lambda_{n}(.,y)\right)\mathbf{1}_{\|y\|\leq n} dy\right\|_{L^1 (\overline{\mu})}\xrightarrow[n \to \infty]{} 0$$.
We have,
\begin{align}
\Pi_{M_n}&=\left\|\int_{D_n}\left(\Lambda(.,y)-\Lambda_{n}(.,y)\right)dy\right\|_{L^1 (\overline{\mu})}= \int_{D}\left|\int_{D} \left(\Lambda(x,y)-\Lambda_{n}(x,y)\right)\mathbf{1}_{\|y\|\leq _{M_n}}dy\right|\overline{\mu}(x)dx,\nonumber\\
\Pi_{n}&\leq \int_{D}\int_{D} G_{n} (x,y) dydx,\quad \text{where } \quad G_{n}(x,y):= \left|\Lambda(x,y)-\Lambda_{n}(x,y)\right|\mathbf{1}_{\|y\|\leq _{M_n}}\overline{\mu}(x)\nonumber.
\end{align}
Morever, by using inequality  $\displaystyle |\phi(a)-\phi(b)|\leq |a-b|\sup_{c\in [a;b]} |\phi'(z)|$ with $\phi(c)=\frac{1}{c^{\gamma}}$,  with \\$a=\displaystyle\int_{D}K(z,y)\mathbf{1}_{\|z\|\leq M_n}\overline{\mu}(dz)$ and $b=\displaystyle\int_{D}K(z,y)\overline{\mu}(dz)$ , we obtain:
\begin{align}
G_{n}(x,y)&\leq \gamma  K(x,y)\overline{\mu}(y)\overline{\mu}(x) \dfrac{\displaystyle\int_{D}K(z,y)\mathbf{1}_{\|z\|>M_n}\overline{\mu}(dz)}{\left[\displaystyle\int_{D}K(z,y)\mathbf{1}_{\|z\|\leq M_n}\overline{\mu}(dz) \right]^{\gamma+1}}\mathbf{1}_{\|y\|\leq M_n},\nonumber\\
&\leq C^2 C_0^3 e^{-a\|y\|^{\delta}}e^{-a\|x\|^{\delta}} \mathbf{1}_{\|x-y\|\leq \underline{R}}\dfrac{\displaystyle\int_{D}\mathbf{1}_{\|z-y\|\leq \underline{R}}\mathbf{1}_{\|z\|>M_n}e^{-a\|z\|^{\delta}}dz}{(\underline {c}c_0)^{\gamma+1}\left[\displaystyle\int_{D}\mathbf{1}_{\|z-y\|\leq r}\mathbf{1}_{\|z\|\leq M_n}e^{-a\|z\|^{\delta}}dz\right]^{\gamma+1}}\mathbf{1}_{\|y\|\leq M_n},\nonumber\\
&\leq Ce^{a\left(\gamma M_n^{\delta}-( M_n-\underline{R})^{\delta}\right)} e^{-a\|x\|^{\delta}} \dfrac{\mathbf{1}_{\|x-y\|\leq \underline{R}} }{\left[\displaystyle\int_{C(y,l_y,\alpha)\cap B(y,r)}dz\right]^{\gamma+1}},\nonumber\\
&\leq  Ce^{a\left(\gamma M_n^{\delta}-( M_n-\underline{R})^{\delta}\right)} e^{-a\|x\|^{\delta}} \mathbf{1}_{\|x-y\|\leq \underline{R}}\nonumber.
\end{align}
Now $\displaystyle \int_{\R^d} \int_{\R^d} e^{-a\|x\|^{\delta}} \mathbf{1}_{\|x-y\|\leq \underline{R}}dxdy<\infty,$ and Lemma \ref{rapp} combined with $M_n=M_n-\underline{R}+\underline{R}$ yields that $\gamma M_n^{\delta}-( M_n-\underline{R})^{\delta}\leq (\gamma-\frac{1}{(1+\epsilon)})M_n^{\delta}+C_{\epsilon,\delta}\underline{R}^{\delta}$. Choosing  $\epsilon>0$ small enough such that $\gamma(1+\epsilon)< 1$,  we conclude that $\Pi_{n} \rightarrow 0$, as  $n\rightarrow \infty$.
\end{proof}
The next step in the proof of Theorem \ref{theo1} will be to establish the convergence result along a subsequence $N_k$. For that sake, we will choose such a subsequence so that 
\[(\overline{\mu}^{S,N_k}_{n,t}, \overline{\mu}_{n,t}^{\mathfrak{F},N_k},\overline{\mu}_{n,t}^{I,N_k},\overline{\mu}_{n,t}^{R,N_k})\Rightarrow
 (\overline{\mu}^{S,N_k}_t, \overline{\mu}_t^{\mathfrak{F},N_k},\overline{\mu}_t^{I,N_k},\overline{\mu}_t^{R,N_k})\]
 locally uniformly in $t$ and uniformly in $k$.
Lemma \ref{Lem imp2} below will be the essential step in our argument. Two other Lemmas and one Corollary will serve as a preparation.
Define\\
$\displaystyle Z_n^N(t)= \displaystyle \frac{1}{N}\sum_{i\in\mathfrak{S}_N}  \int_{0}^{t}\int_0^{\infty} \mathbf{1}_{\min \left(\overline{\Gamma}_{n}^{N}(v,X^i);\overline{\Gamma}^{N}(v,X^i)\right)\leq u\leq \max \left(\overline{\Gamma}_{n}^{N}(v,X^i);\overline{\Gamma}^{N}(v,X^i)\right)} \mathbf{1}_{\|X^i\|\leq M_n}\overline{P}^{i}(dv,du)$,\\
$\displaystyle Y^N_n(t)=\displaystyle\frac{1}{N}\sum_{i\in\mathfrak{S}_N}  \int_{0}^{t}\int_{0}^{\infty}\left|\mathbf{1}_{A^N_{n,i}(v^{-}) = 0}-\mathbf{1}_{A^N_{i}(v^{-}) = 0}\right|\mathbf{1}_{u\leq \overline{\Gamma}_{n}^{N}(v,X^i)}\mathbf{1}_{\|X^i\|\leq M_n}\overline{P}^{i}(dv,du)$.
\begin{lemma}
\label{lem}
Let $T>0.$ For any $n\geq 1$, as $N\rightarrow \infty$,  $\displaystyle \sup_{t\in[0;T]}Z_n^N(t)$  and $\displaystyle \sup_{t\in[0;T]}Y_n^N(t)$ converge in probability  to 0.
\end{lemma}
\begin{proof} We have
\begin{align*}
 \inf_{1\leq \ell\leq q(M_n)}\overline{\mu}(\mathfrak{A}_\ell)=  \inf_{1\leq k\leq q(M_n)} \int_{\mathfrak{A}_\ell} e^{-a\|u\|^\delta}du&\geq c \inf_{1\leq \ell\leq q(M_n); u\in\mathfrak{A}_\ell } e^{-a\|u\|^\delta}, \\
 &\geq ce^{-aM_n^\delta}=:a(n).\\
  \inf_{1\leq \ell\leq q(M_n+\underline{R})}\overline{\mu}(\mathfrak{A}_\ell)=  \inf_{1\leq \ell\leq q(M_n+\underline{R})} \int_{\mathfrak{A}_\ell} e^{-a\|u\|^\delta}du&\geq c \inf_{1\leq \ell\leq q(M_n+\underline{R}); u\in\mathfrak{A}_\ell } e^{-a\|u\|^\delta}\\
 &\geq ce^{-a(\underline{R}+M_n)^\delta}=: b(n).
\end{align*}
Let us consider:
\begin{align*}
\Omega_0^N:= \left\{ \inf_{1\leq \ell\leq q(M_n)}\overline{\mu}^N(\mathfrak{A}_\ell)\geq \frac{a(n)}{2} \right\} ,\quad \text{and}\quad \Omega_1^N:= \left\{ \inf_{1\leq \ell\leq q(M_n+\underline{R})}\overline{\mu}^N(\mathfrak{A}_\ell)\geq \frac{b(n)}{2} \right\} .
\end{align*}
\begin{align*}
\forall \ell\in\{1;\cdots; q(M_n)\}, \quad \overline{\mu}^N(\mathfrak{A}_\ell ) \xrightarrow[N \to \infty]{} \overline{\mu}(\mathfrak{A}_\ell )&\implies  \inf_{1\leq k\leq q(M_n)}\overline{\mu}^N(\mathfrak{A}_\ell ) \xrightarrow[N \to \infty]{}  \inf_{1\leq \ell\leq q(M_n)}\overline{\mu}(\mathfrak{A}_\ell ),\\
\forall \ell\in\{1;\cdots; q(M_n+\underline{R})\}, \quad \overline{\mu}^N(\mathfrak{A}_\ell ) \xrightarrow[N \to \infty]{} \overline{\mu}(\mathfrak{A}_\ell )&\implies  \inf_{1\leq \ell\leq q(M_n+\underline{R})}\overline{\mu}^N(\mathfrak{A}_\ell ) \xrightarrow[N \to \infty]{}  \inf_{1\leq \ell\leq q(M_n+\underline{R})}\overline{\mu}(\mathfrak{A}_\ell ).
\end{align*}
Consequently, $\displaystyle\mathbb{P}\left[\left( \Omega_0^N\right)^c\right] \xrightarrow[N \to \infty]{} 0$ and $\displaystyle\mathbb{P}\left[\left( \Omega_1^N\right)^c\right] \xrightarrow[N \to \infty]{} 0$.\\
\\
On the one hand,
\begin{align*}
\mathbb{E}\left[\left(\sup_{t\in[0;T]}Z^N_n(t))\right)^2 \wedge 1\right]&\leq \mathbb{P}\left[\left( \Omega_0^N \cap \Omega_1^N\right)^c\right]\\
&+ 4 \mathbb{E}\left[\frac{1}{N^2}\sum_{i\in\mathfrak{S}_N}\int_{0}^{T}\mathbf{1}_{\Omega_0^N \cap \Omega_1^N}\left|\overline{\Gamma}_{n}^{N}(s,X^i)-\overline{\Gamma}^{N}(s,X^i)\right|\mathbf{1}_{\|X^i\|\leq M_n}ds\right],\\
&\leq \mathbb{P}\left[\left( \Omega_0^N \right)^c\right]+\mathbb{P}\left[\left( \Omega_1^N \right)^c\right]\\
&+\frac{4}{N^2}\sum_{i\in\mathfrak{S}_N}\int_{0}^{T} \mathbb{E}\left[ \left(\mathbf{1}_{\Omega_0^N \cap \Omega_1^N}\overline{\Gamma}_{n}^{N}(s,X^i)+\mathbf{1}_{\Omega_0^N \cap \Omega_1^N} \overline{\Gamma}^{N}(s,X^i)\right)\mathbf{1}_{\|X^i\|\leq M_n}\right]ds.
\end{align*}
\begin{align*}
\overline{\Gamma}_{n}^{N}(s,X^i)&= \frac{1}{N} \displaystyle \sum_{j\in\mathfrak{I}_N} \frac{K(X^{i},X^{j})}{\left[\displaystyle \int_{D_n} K(z,X^j)\overline{\mu}^N(dz)\right]^{\gamma}} \lambda_{-j}(s) \mathbf{1}_{\|X^j\|\leq M_n}\mathbf{1}_{\|X^i\|\leq M_n}\\
&+\frac{1}{N} \sum_{j\in\mathfrak{S}_N} \frac{K(X^{i},X^{j})}{\left[\displaystyle \int_D K(z,X^j)\mathbf{1}_{\|z\|\leq M_n}\overline{\mu}^N(dz)\right]^{\gamma}}\lambda_{j}(s- \tau_{n}^N(j)) \mathbf{1}_{\|X^j\|\leq M_n}\mathbf{1}_{\|X^i\|\leq M_n},\\
&\leq \frac{C}{N} \displaystyle \sum_{j\in\mathfrak{I}_N} \frac{\mathbf{1}_{\|X^i-X^j\|\leq \underline{R}}}{\left[\displaystyle \int_{C(X^j,l_{X^j},\alpha)\cap B(X^j,r)} \overline{\mu}^N(dz)\right]^{\gamma}}  \mathbf{1}_{\|X^j\|\leq M_n}\\
&+\frac{C}{N} \displaystyle \sum_{j\in\mathfrak{S}_N} \frac{\mathbf{1}_{\|X^i-X^j\|\leq \underline{R}}}{\left[\displaystyle \int_{C(X^j,l_{X^j},\alpha) \cap B(X^j,r)} \overline{\mu}^N(dz)\right]^{\gamma}}  \mathbf{1}_{\|X^j\|\leq M_n},\\
\overline{\Gamma}_{n}^{N}(s,X^i)&\leq  \frac{C}{\left[\displaystyle  \inf_{1\leq k\leq q(M_n)}\overline{\mu}^N(\mathfrak{A}_\ell)\right]^\gamma} \\
\overline{\Gamma}^{N}(s,X^i)\mathbf{1}_{\|X^i\|\leq M_n}&= \frac{1}{N} \displaystyle \sum_{j\in\mathfrak{I}_N} \frac{K(X^{i},X^{j})}{\left[\displaystyle \int_{D} K(z,X^j)\overline{\mu}^N(dz)\right]^{\gamma}} \lambda_{-j}(s)\mathbf{1}_{\|X^i\|\leq M_n} \\
&+\frac{1}{N} \sum_{j\in\mathfrak{S}_N} \frac{K(X^{i},X^{j})}{\left[\displaystyle \int_D K(z,X^j)\overline{\mu}^N(dz)\right]^{\gamma}}\lambda_{j}(s- \tau^N(j)) \mathbf{1}_{\|X^i\|\leq M_n},\\
&\leq \frac{C}{N} \displaystyle \sum_{j\in\mathfrak{I}_N} \frac{\mathbf{1}_{\|X^j\|\leq \underline{R}+M_n}}{\left[\displaystyle \int_{B(X^j,l_{X^j},\alpha) \cap B(X^j,r)} \overline{\mu}^N(dz)\right]^{\gamma}} \mathbf{1}_{\|X^i\|\leq M_n} \\
&+\frac{C}{N} \displaystyle \sum_{j\in\mathfrak{S}_N} \frac{\mathbf{1}_{\|X^j\|\leq \underline{R}+M_n}}{\left[\displaystyle \int_{B(X^j,l_{X^j},\alpha) \cap B(X^j,r)} \overline{\mu}^N(dz)\right]^{\gamma}}  \mathbf{1}_{\|X^i\|\leq M_n},\\
\overline{\Gamma}^{N}(s,X^i)\mathbf{1}_{\|X^i\|\leq M_n}&\leq  \frac{C}{\left[\displaystyle  \inf_{1\leq \ell\leq q(M_n+\underline{R})}\overline{\mu}^N(\mathfrak{A}_\ell)\right]^\gamma} 
\end{align*}

\begin{align*}
\mathbb{E}\left[ \left(\mathbf{1}_{\Omega_0^N \cap \Omega_1^N}\overline{\Gamma}_{n}^{N}(s,X^i)+\mathbf{1}_{\Omega_0^N \cap \Omega_1^N} \overline{\Gamma}^{N}(s,X^i)\right)\mathbf{1}_{\|X^i\|\leq M_n}\right]&\leq C\left(\frac{a(n)}{2}\right)^{-\gamma}+C\left(\frac{b(n)}{2}\right)^{-\gamma}\\
&=C(n).
\end{align*}
\begin{align*}
\mathbb{E}\left[\left(\sup_{t\in[0;T]}Z^N_n(t)\right)^2 \wedge 1\right]&\leq  \mathbb{P}\left[\left( \Omega_0^N \right)^c\right]+\mathbb{P}\left[\left( \Omega_1^N \right)^c\right]+ \frac{C(n)T}{N} \xrightarrow[N \to \infty]{} 0
\end{align*}
On the other hand,
\begin{align*}
\mathbb{E}\left[\left(\sup_{t\in[0;T]}Y^N_n(t)\right)^2 \wedge 1\right]&\leq \frac{4}{N^2}\sum_{i\in\mathfrak{S}_N} \int_{0}^{T}\mathbb{E}\left[\mathbf{1}_{ \Omega_0^N}\left|\mathbf{1}_{A^N_{n,i}(s^{-}) = 0}-\mathbf{1}_{A^N_{i}(s^{-}) = 0}\right| \overline{\Gamma}_{n}^{N}(s,X^i)\mathbf{1}_{\|X^i\|\leq M_n}ds\right]\\
&+ \mathbb{P}\left[\left( \Omega_0^N\right)^c\right]\\
&\leq  \mathbb{P}\left[\left( \Omega_0^N\right)^c\right]+ \frac{C\times (a(n))^{-\gamma} T}{N}\xrightarrow[N \to \infty]{} 0.
\end{align*}
We conclude that  $\displaystyle \sup_{t\in[0;T]} Y^N_n(t)$ and $\displaystyle \sup_{t\in[0;T]}Z^N_n(t)$ converge in probability to 0.
\end{proof}
\begin{coro}
\label{coro}
Given $T>0$, for any subsequence of the sequence $\{N=1,2,\ldots\}$, we can extract a further subsequence $\{N_k,\ k\ge1\}$ such that 
for any $n\ge1$, $(Y^{N_k}_n(t),Z^{N_k}_n(t))\to(0,0)$ as $k\to\infty$ a.s., uniformly for $t\in[0,T]$.
\end{coro}
\begin{proof}
Thanks to Lemma \ref{lem}, we  first extract a subsequence $\{N_{1,k},\ k\ge1\}$ such that, as $k\to\infty$, $(Y^{N_{1,k}}_1(t),Z^{N_{1,k}}_1(t))\to(0,0)$ a.s., uniformly for $t\in[0,T]$.
We next extract from $\{N_{1,k},\ k\ge1\}$ a second subsequence $\{N_{2,k},\ k\ge1\}$ such that, as $k\to\infty$, $(Y^{N_{2,k}}_2(t),Z^{N_{2,k}}_2(t))\to(0,0)$  a.s., uniformly for $t\in[0,T]$. We do that successively for $n=3, 4,\ldots$, so that along the subsequence $\{N_{n,k},\ k\ge1\}$
(extracted from the subsequence $\{N_{n-1,k},\ k\ge1\}$), as $k\to\infty$, $(Y^{N_{n,k}}_n(t),Z^{N_{n,k}}_n(t))\to(0,0)$  a.s., uniformly for $t\in[0,T]$.

Finally we define the diagonal sequence $\{N_k:=N_{k,k},\ k\ge1\}$, along which it is clear that for any $n\ge1$, $(Y^{N_k}_n(t),Z^{N_k}_n(t))\to(0,0)$ as $k\to\infty$ a.s., uniformly for $t\in[0,T]$. Indeed, for $k\ge n$, the sequence $N_k$ is a subsequence of $N_{n,k}$.
\end{proof}
\bigskip

We consider the quantity
\[\frac{1}{N}\sum_{i\in\mathfrak{S}_N}\left(\inf_{k\le q(\|X_i\|+\underline{R})}\bar{\mu}^N(\mathfrak{A}_k)\right)^{-\gamma}
{\bf1}_{\|X_i\|\le M_n}=\int_{B(0,M_n)}\left(\inf_{k\le q(\|x\|+\underline{R})}\bar{\mu}^N(\mathfrak{A}_k)\right)^{-\gamma}\bar{\mu}^{S,N}_{0}(dx)\,.\]

We have the following result.
\begin{lemma}
\label{lemEP}
As $N\to\infty$,
\[\int_{B(0,M_n)}\left(\inf_{k\le q(\|x\|+\underline{R})}\bar{\mu}^N(\mathfrak{A}_k)\right)^{-\gamma}\bar{\mu}^{S,N}_0(dx)
\to\int_{B(0,M_n)}\left(\inf_{k\le q(\|x\|+\underline{R})}\bar{\mu}(\mathfrak{A}_k)\right)^{-\gamma}\bar{\mu}^{S}_0(dx)\ \text{ a.s.}\]
Moreover, 
$\displaystyle \sup_n\int_{B(0,M_n)}\left(\inf_{k\le q(\|x\|+\underline{R})}\bar{\mu}(\mathfrak{A}_k)\right)^{-\gamma}\bar{\mu}^{S}_0(dx)<\infty$.
\end{lemma}
\begin{proof}
We have
\begin{align*}
&\left|\int_{B(0,M_n)}\left(\inf_{k\le q(\|x\|+\underline{R})}\bar{\mu}(\mathfrak{A}_k)\right)^{-\gamma}\bar{\mu}^{S}_0(dx)-
\int_{B(0,M_n)}\left(\inf_{k\le q(\|x\|+\underline{R})}\bar{\mu}^N(\mathfrak{A}_k)\right)^{-\gamma}\bar{\mu}^{S,N}_0(dx)\right|\\
&\quad \le\left|\int_{B(0,M_n)}\left(\inf_{k\le q(\|x\|+\underline{R})}\bar{\mu}(\mathfrak{A}_k)\right)^{-\gamma}\left[\bar{\mu}^{S}_0(dx)-
\bar{\mu}^{S,N}_0(dx)\right]\right|\\
&\quad+\left|\int_{B(0,M_n)}\left[\left(\inf_{k\le q(\|x\|+\underline{R})}\bar{\mu}(\mathfrak{A}_k)\right)^{-\gamma}-
\left(\inf_{k\le q(\|x\|+\underline{R})}\bar{\mu}^N(\mathfrak{A}_k)\right)^{-\gamma}\right]\bar{\mu}^{S,N}_0(dx)\right|
\end{align*}
Since  as $N\to\infty$, $\bar{\mu}^{S,N}_0(\cdot)\Rightarrow\bar{\mu}^{S}_0(\cdot)$ and 
$x\mapsto\left(\inf_{k\le q(\|x\|+\underline{R})}\bar{\mu}(\mathfrak{A}_k)\right)^{-\gamma}$ is bounded on $B(0,M_n)$ and 
$\bar{\mu}^{S}_0(\cdot)$ a.e. continuous, the first term on the right tends to $0$ as $N\to\infty$. The convergence to 
$0$ of the second term follows from the facts that 
\[ \sup_{\|x\|\le M_n}\left|\left(\inf_{k\le q(\|x\|+\underline{R})}\bar{\mu}(\mathfrak{A}_k)\right)^{-\gamma}-
\left(\inf_{k\le q(\|x\|+\underline{R})}\bar{\mu}^N(\mathfrak{A}_k)\right)^{-\gamma}\right|\to0,\]
as $N\to\infty$ (this is a consequence of the fact that $(\bar{\mu}^N(\mathfrak{A}_1),\ldots,\bar{\mu}^N(\mathfrak{A}_{q(M_n)}))\to
(\bar{\mu}(\mathfrak{A}_1),\ldots,\bar{\mu}(\mathfrak{A}_{q(M_n)}))$ a.s. as $N\to\infty$), while the total mass of $\bar{\mu}^{S,N}_0(\cdot)$ is less than or equal to $1$ a.s. for all $N$.

We finally prove the second part of the statement.
\[ \left(\inf_{k\le q(\|x\|+\underline{R})}\bar{\mu}(\mathfrak{A}_k)\right)^{-\gamma}\le Cc^{-\gamma}e^{a\gamma(\|x\|+\underline{R})^{\delta}},\]
while
\[\bar{\mu}^{S,N}_n(t,dx)\le \bar{S}(0,x)\bar{\mu}(x)dx\le Ce^{-a\|x\|^\delta}dx\,.\]
Consequently
\begin{align*}
\sup_n\int_{B(0,M_n)}\left(\inf_{k\le q(\|x\|+\underline{R})}\bar{\mu}(\mathfrak{A}_k)\right)^{-\gamma}\bar{\mu}^{S}_0(dx)
&\le C\int_{\R^d}e^{a\gamma(\|x\|+\underline{R})^{\delta}}e^{-a\|x\|^\delta}dx
\end{align*}
It remains to show that the right hand side is finite, which follows from the fact that $\gamma<1$. 
Indeed, the integrand is locally bounded, and choosing any $\gamma'\in(\gamma,1)$, on the complement of the ball centered at $0$ with radius
$R\left((\gamma'/\gamma)^{1/\delta}-1\right)^{-1}$, the integrand is upper bounded by $\exp(-a(1-\gamma')\|x\|^\delta)$, which clearly is integrable on $\R^d$.
\end{proof}

We can now establish our main technical result.

\begin{lemma}
	\label{Lem imp2}
	For any $T>0 $, as $n \rightarrow \infty$,
	\[\mathbb{E}\left[\displaystyle\limsup_{k}\left(\displaystyle\frac{1}{N_k}\sum_{i\in\mathfrak{S}_N}\sup_{t\in[0;T]}\left| A^{N_k}_{n,i}(t)-A^{N_k}_{i}(t)\right|\mathbf{1}_{\|X^i\|\leq M_n}\right)\right]\rightarrow 0\,.\]
\end{lemma}
\begin{proof}	
We fix $T>0$ arbitrary. For any $t\in[0;T]$,
	\begin{align}
	A^{N_k}_{n,i}(s)-A^{N_k}_{i}(s)&=\int_{0}^{s}\int_{0}^{\infty}\left(\mathbf{1}_{A^{N_k}_{n,i}(v^{-}) = 0}\mathbf{1}_{u\leq \overline{\Gamma}_{n}^{N_k}(v,X^i)}-\mathbf{1}_{A^{N_k}_{i}(v^{-}) = 0}\mathbf{1}_{u\leq \overline{\Gamma}^{N_k}(v,X^i)}\right)P^{i}(dv,du)\nonumber\\
	&=\int_{0}^{s}\int_{0}^{\infty}\mathbf{1}_{A^{N_k}_{i}(v^{-}) = 0}\left(\mathbf{1}_{u\leq \overline{\Gamma}_{n}^{N_k}(v,X^i)}- \mathbf{1}_{u\leq \overline{\Gamma}^{N_k}(v,X^i)}\right)P^{i}(dv,du)\nonumber\\ 
	&+\int_{0}^{s}\int_{0}^{\infty}\left(\mathbf{1}_{A^{N_k}_{n,i}(v^{-}) = 0}-\mathbf{1}_{A^{N_k}_{i}(v^{-}) = 0}\right)\mathbf{1}_{u\leq \overline{\Gamma}_{n}^{N_k}(v,X^i)}P^{i}(dv,du)\nonumber\\
\sup_{s\in[0;t]}\left|A^{N_k}_{n,i}(s)-A^{N_k}_{i}(s) \right|&\leq \int_{0}^{t}\int_{\min \left(\overline{\Gamma}_{n}^{N_k}(v,X^i);\overline{\Gamma}^{N_k}(v,X^i)\right)}^{ \max \left(\overline{\Gamma}_{n}^{N_k}(v,X^i);\overline{\Gamma}^{N_k}(v,X^i)\right)}P^{i}(dv,du)\nonumber\\
	&+\int_{0}^{t}\int_{0}^{\infty} \left| A^{N_k}_{n,i}(v)-A^{N_k}_{i}(v) \right|\mathbf{1}_{u\leq \overline{\Gamma}_{n}^{N_k}(v,X^i)}P^{i}(dv,du)\nonumber
	\end{align}
	\begin{align}
		\label{lem7eq1}
	\mathbb{E}\left[\displaystyle\limsup_{k}\left(\displaystyle\frac{1}{N_k}\sum_{i\in\mathfrak{S}_{N_k}}\sup_{s\in[0;t]}\left| A^{N_k}_{n,i}(s)-A^{N_k}_{i}(s)\right|\mathbf{1}_{\|X^i\|\leq M_n}\right)\right]&\leq  \Psi^{N_k}_1(n,t)+ \Psi^{N_k}_2(n,t)+\Psi^{N_k}_3(n,t)+\Psi^{N_k}_4(n,t);
	\end{align}
	where 
	
	\begin{align*}
\Psi^{N_k}_1(n,t)&= \int_0^t \mathbb{E}\left[\displaystyle\limsup_{k}\left(\displaystyle\frac{1}{N_k}\sum_{i\in\mathfrak{S}_{N_k}}\left| \overline{\Gamma}_{n}^{N_k}(s,X^i)-\overline{\Gamma}^{N_k}(s,X^i)\right|\mathbf{1}_{\|X^i\|\leq M_n}\right)\right]ds,\\
\Psi^{N_k}_2(n,t)&=\int_0^t \mathbb{E}\left[\displaystyle\limsup_{k}\left(\displaystyle\frac{1}{N_k}\sum_{i\in\mathfrak{S}_{N_k}}\sup_{v\in[0;s]}\left| A^{N_k}_{n,i}(v)-A^{N_k}_{i}(v) \right| \overline{\Gamma}_{n}^{N_k}(v,X^i)\mathbf{1}_{\|X^i\|\leq M_n}\right)\right]ds,\\
\Psi^{N_k}_3(n,t)&=\mathbb{E}\left[\displaystyle\limsup_{k}\left(\displaystyle\frac{1}{N_k}\!\!\sum_{i\in\mathfrak{S}_{N_k}} \!\! \int_{0}^{t}\!\!\int_0^{\infty}\!\!\!\! \mathbf{1}_{\min \left(\overline{\Gamma}_{n}^{N_k}(v,X^i);\overline{\Gamma}^{N_k}(v,X^i)\right)\leq u\leq \max \left(\overline{\Gamma}_{n}^{N_k}(v,X^i);\overline{\Gamma}^{N_k}(v,X^i)\right)} \mathbf{1}_{\|X^i\|\leq M_n}\overline{P}^{i}(dv,du)\right)\right],\\
\Psi^{N_k}_4(n,t)&=\mathbb{E}\left[\displaystyle\limsup_{k}\left(\displaystyle\frac{1}{N_k}\sum_{i\in\mathfrak{S}_{N_k}}  \int_{0}^{t}\int_{0}^{\infty}\left|\mathbf{1}_{A^{N_k}_{n,i}(v^{-}) = 0}-\mathbf{1}_{A^{N_k}_{i}(v^{-}) = 0}\right|\mathbf{1}_{u\leq \overline{\Gamma}_{n}^{N_k}(v,X^i)}\mathbf{1}_{\|X^i\|\leq M_n}\overline{P}^{i}(dv,du)\right)\right].
\end{align*}
According to Corollary \ref{coro}, for any $n\ge1$,
  \begin{align}
\label{lem7eq10}
  \Psi^{N_k}_3(n,t)= \Psi^{N_k}_4(n,t)=0\,.
  \end{align}
\\

Let $\displaystyle \overline{B}_n^{N_k}(s,X^i):= \left| \overline{\Gamma}_{n}^{N_k}(s,X^i)-\overline{\Gamma}^{N_k}(s,X^i)\right|$ and  $\displaystyle \overline{\Upsilon}_n^{i,N_k}(s,X^i):=\left| A^{N_k}_{n,i}(s)-A^{N_k}_{i}(s) \right| \overline{\Gamma}_{n}^{N_k}(s,X^i)$.
On the one hand,
\begin{align}
\overline{B}_n^{N_k}(s,X^i )&\leq  \frac{1}{N_k} \displaystyle  \sum_{j\in\mathfrak{I}_{N_k}} \left|\frac{K(X^{i},X^{j})\lambda_{-j}(s)\mathbf{1}_{\|X^j\|\leq M_n}}{\left[\displaystyle \int_{D_n} K(z,X^j)\overline{\mu}^{N_k}(dz)\right]^{\gamma}}- \frac{K(X^{i},X^{j})\lambda_{-j}(s)}{\left[\displaystyle \int_D K(z,X^j)\overline{\mu}^{N_k}(dz)\right]^{\gamma}}\right| \nonumber\\
&+\frac{1}{N_k}\sum_{j\in\mathfrak{S}_{N_k}} \left|\frac{K(X^{i},X^{j})\lambda_{j}(s- \tau_{M_n}^{N_k}(j))\mathbf{1}_{\|X^j\|\leq M_n}}{\left[\displaystyle \int_{D_n} K(z,X^j)\overline{\mu}^{N_k}(dz)\right]^{\gamma}}-\frac{K(X^{i},X^{j})\lambda_{j}(s- \tau^{N_k}(j))}{\left[\displaystyle \int_D K(z,X^j)\overline{\mu}^{N_k}(dz)\right]^{\gamma}}\right|\nonumber\\
&\leq \frac{\lambda^*}{N_k} \displaystyle  \sum_{j\in\mathfrak{I}_{N_k}} \frac{K(X^{i},X^{j})\displaystyle \int_D K(z,X^j)\mathbf{1}_{\|z\|> M_n}\overline{\mu}^{N_k}(dz)}{\left[\displaystyle \int_{D_n} K(z,X^j)\overline{\mu}^{N_k}(dz)\right]^{\gamma+1}}\mathbf{1}_{\|X^j\|\leq M_n}\nonumber\\
&+\frac{\lambda^*}{N_k} \displaystyle  \sum_{j\in\mathfrak{S}_{N_k}} \frac{K(X^{i},X^{j})\displaystyle \int_D K(z,X^j)\mathbf{1}_{\|z\|> M_n}\overline{\mu}^{N_k}(dz)}{\left[\displaystyle \int_{D_n} K(z,X^j)\overline{\mu}^{N_k}(dz)\right]^{\gamma+1}}\mathbf{1}_{\|X^j\|\leq M_n}\nonumber\\
&+\frac{\lambda^*}{N_k} \displaystyle  \sum_{j\in\mathfrak{I}_{N_k}}\frac{K(X^{i},X^{j})}{\left[\displaystyle \int_D K(z,X^j)\overline{\mu}^{N_k}(dz)\right]^{\gamma}}\mathbf{1}_{\|X^j\|>M_n}\nonumber\\
&+\frac{\lambda^*}{N_k} \displaystyle  \sum_{j\in\mathfrak{S}_{N_k}}\frac{K(X^{i},X^{j})}{\left[\displaystyle \int_D K(z,X^j)\overline{\mu}^{N_k}(dz)\right]^{\gamma}}\mathbf{1}_{\|X^j\|>M_n}\nonumber\\
&+\frac{\lambda^*}{N_k}\sum_{j\in\mathfrak{S}_{N_k}} \frac{K(X^{i},X^{j})\mathbf{1}_{\|X^j\|\leq M_n}}{\left[\displaystyle \int_{D_n} K(z,X^j)\overline{\mu}^{N_k}(dz)\right]^{\gamma}}\left| A_{n,j}^{N_k}(s)-A_{j}^{N_k}(s)\right|\nonumber\\
&\leq \lambda^* \displaystyle \int_D K(X^{i},y)\frac{\displaystyle \int_D K(z,y)\mathbf{1}_{\|z\|> M_n}\overline{\mu}^{N_k}(dz)}{\left[\displaystyle \int_{D_n} K(z,y) \overline{\mu}^{N_k}(dz)\right]^{\gamma+1}}\mathbf{1}_{\|y\|\leq M_n}\overline{\mu}^{N_k}(dy)\nonumber\\
&+ \lambda^* \displaystyle \int_D \frac{K(X^{i},y)}{\left[\displaystyle \int_D K(z,y)\overline{\mu}^{N_k}(dz)\right]^{\gamma}}\mathbf{1}_{\|y\|>M_n}\overline{\mu}^{N_k}(dy)\nonumber\\
&+\frac{\lambda^*}{N_k}\sum_{j\in\mathfrak{S}_{N_k}} \frac{K(X^{i},X^{j}) \mathbf{1}_{\|X^j\|\leq M_n}}{\left[\displaystyle \int_{D_n} K(z,X^j)\overline{\mu}^{N_k}(dz)\right]^{\gamma}}\left| A_{n,j}^{N_k}(s)-A_{j}^{N_k}(s)\right|\nonumber\\
&\leq C \displaystyle \int_D \mathbf{1}_{\|X^i-y\|\leq \underline{R}}\frac{\displaystyle \int_D \mathbf{1}_{\|z-y\|\leq \underline{R}}\mathbf{1}_{\|z\|> M_n}\overline{\mu}^{N_k}(dz)}{\left[\displaystyle \int_{C(y,l_y,\alpha)\cap B(y,r)} \overline{\mu}^{N_k}(dz)\right]^{\gamma+1}}\mathbf{1}_{\|y\|\leq M_n}\overline{\mu}^{N_k}(dy)\nonumber\\
&+C \displaystyle \int_D \frac{\mathbf{1}_{\|X^i-y\|\leq \underline{R}}}{\left[\displaystyle \int_{C(y,l_y,\alpha)\cap B(y,r)} \overline{\mu}^{N_k}(dz)\right]^{\gamma}}\mathbf{1}_{\|y\|>M_n}\overline{\mu}^{N_k}(dy)\nonumber\\
&+ C\frac{1}{N_k}\sum_{j\in\mathfrak{S}_{N_k}} \frac{\mathbf{1}_{\|X^i-X^j\|\leq \underline{R} }\mathbf{1}_{\|X^j\|\leq M_n}}{\left[\displaystyle \int_{C(X^j,l_{X^j},\alpha)\cap B(X^j,r)} \overline{\mu}^{N_k}(dz)\right]^{\gamma}}\left| A_{n,j}^{N_k}(s)-A_{j}^{N_k}(s)\right|\nonumber\\
\overline{B}_n^{N_k}(s,X^i )\mathbf{1}_{\|X^i\|\leq M_n}&\leq C \displaystyle \int_D \mathbf{1}_{\|X^i-y\|\leq \underline{R}}\frac{\displaystyle \int_D \mathbf{1}_{\|z-y\|\leq \underline{R}}\mathbf{1}_{\|z\|> M_n}\overline{\mu}^{N_k}(dz)}{\left[\displaystyle \int_{C(y,l_y,\alpha)\cap B(y,r)} \overline{\mu}^{N_k}(dz)\right]^{\gamma+1}}\mathbf{1}_{\|y\|\leq M_n}\overline{\mu}^{N_k}(dy) \nonumber\\
&+C \displaystyle \int_D \frac{\mathbf{1}_{\|X^i-y\|\leq \underline{R}}\mathbf{1}_{\|y\|\leq M_n +\underline{R}}}{\left[\displaystyle \int_{C(y,l_y,\alpha)\cap B(y,r)} \overline{\mu}^{N_k}(dz)\right]^{\gamma}}\mathbf{1}_{\|y\|>M_n}\overline{\mu}^{N_k}(dy)\nonumber\\
&+ C\frac{1}{N_k}\sum_{j\in\mathfrak{S}_{N_k}} \frac{\mathbf{1}_{\|X^j\|\leq \underline{R}+\|X^i\|} \mathbf{1}_{\|X^i\|\leq M_n}}{\left[\displaystyle \int_{C(X^j,l_{X^j},\alpha)\cap B(X^j,r)} \overline{\mu}^{N_k}(dz)\right]^{\gamma}}\left| A_{n,j}^{N_k}(s)-A_{j}^{N_k}(s)\right|\mathbf{1}_{\|X^j\|\leq M_n}\nonumber\\
&=\displaystyle \Sigma_1^{N_k}(X^i,n)+\Sigma_2^{N_k}(X^i,n)+\Sigma_3^{N_k}(X^i,n,s)\nonumber
\end{align}
\bigskip
We need to estimate 
\[ \E\left[\limsup_k\left(\frac{1}{N_k}\sum_{i\in\mathfrak{S}_{N_k}}\overline{B}^{N_k}_n(s,X^i){\bf1}_{\|X^i\|\le M_n}\right)\right]\,.\]
For that sake, we shall estimate
\[\E\left[\limsup_k\left(\frac{1}{N_k}\sum_{i\in\mathfrak{S}_{N_k}}\Sigma_j^{N_k}(X^i,n)\mathbf{1}_{\|X^i\|\leq M_n}\right)\right],\]
successively for $j=1, 2, 3$.

By using Lemma \ref{impor},
\begin{align*}
\Sigma_1^{N_k}(X^i,n)&\leq\frac{C}{(\displaystyle \inf_{1\leq \ell\leq q(M_n)}\overline{\mu}^N(\mathfrak{A}_\ell))^{\gamma+1}}\displaystyle \int_{D}\int_{D}\displaystyle \mathbf{1}_{\|X^i-y\|\leq \underline{R}}\mathbf{1}_{\|z-y\|\leq \underline{R}}\mathbf{1}_{\|z\|> M_n}\overline{\mu}^{N_k}(dz)\overline{\mu}^{N_k}(dy),\\
\frac{1}{N_k}\sum_{i\in\mathfrak{S}_{N_k}}\Sigma_1^{N_k}(X^i,n)\mathbf{1}_{\|X^i\|\leq M_n}&\leq \frac{C}{(\displaystyle \inf_{1\leq \ell \leq q(M_n)}\overline{\mu}^{N_k}(\mathfrak{A}_\ell))^{\gamma+1}}\displaystyle \int_{D^3}\displaystyle \mathbf{1}_{\|x-y\|\leq \underline{R}}\mathbf{1}_{\|z-y\|\leq \underline{R}}\mathbf{1}_{\|z\|> M_n}\overline{\mu}^{N_k}(dz)\overline{\mu}^{N_k}(dy) \overline{\mu}_0^{S,N_k}(dx)
\end{align*}
\begin{align}
 \mathbb{E}\left[\displaystyle\limsup_{k}\left(\displaystyle\frac{1}{N_k}\sum_{i\in\mathfrak{S}_{N_k}}\Sigma_1^{N_k}(X^i,n)\mathbf{1}_{\|X^i\|\leq M_n}\right)\right]& \leq\frac{Ce^{-aM_n^{\delta}}e^{-a(M_n-\underline{R})^{\delta}}}{(\displaystyle \inf_{1\leq \ell\leq q(M_n)}\overline{\mu}(\mathfrak{A}_\ell))^{\gamma+1}}\displaystyle \int_{\R^{3d}}\displaystyle \mathbf{1}_{\|x-y\|\leq \underline{R}}\mathbf{1}_{\|z-y\|\leq \underline{R}}e^{-a\|x\|^{\delta}}dzdydx\nonumber\\
 &\leq \frac{Ce^{-aM_n^{\delta}}e^{-a(M_n-\underline{R})^{\delta}}}{(\displaystyle  \inf_{1\leq \ell \leq q(M_n);u\in\mathfrak{A}_\ell} c e^{-a\|u\|^{\delta}})^{\gamma+1}}\displaystyle \left(\int_{B(0;\underline{R})}du\right)^2 \displaystyle \int_{\R^d}e^{-a\|x\|^{\delta}}dx\nonumber\\
 &\leq C e^{a\gamma M_n^{\delta}}e^{-a(M_n-\underline{R})^{\delta}}\nonumber\\
 \label{lem7eq3}
  \mathbb{E}\left[\displaystyle\limsup_{k}\left(\displaystyle\frac{1}{N_k}\sum_{i\in\mathfrak{S}_{N_k}}\Sigma_1^{N_k}(X^i,n)\mathbf{1}_{\|X^i\|\leq M_n}\right)\right]&\leq C e^{-a(\frac{1}{1+\epsilon}-\gamma)M_n^{\delta}},
\end{align}
which tends to $0$ as $n\to\infty$, provided $1+\epsilon<\gamma^{-1}$.
\begin{align*}
\Sigma_2^{N_k}(X^i,n)&\leq \frac{C}{\left(\displaystyle \inf_{1\leq \ell \leq q(M_n+\underline{R})}\overline{\mu}^N(\mathfrak{A}_\ell)\right)^{\gamma}}\displaystyle \int_D\mathbf{1}_{\|X^i-y\|\leq \underline{R}}\mathbf{1}_{\|y\|\leq M_n +\underline{R}}\mathbf{1}_{\|y\|>M_n}\overline{\mu}^{N_k}(dy),\\
 \frac{1}{N_k}\sum_{i\in\mathfrak{S}_{N_k}}\Sigma_2^{N_k}(X^i,n)\mathbf{1}_{\|X^i\|\leq M_n}&\leq  \frac{C}{\left(\displaystyle \inf_{1\leq \ell \leq q(M_n+\underline{R})}\overline{\mu}^{N_k}(\mathfrak{A}_\ell)\right)^{\gamma}}\displaystyle \int_D \int_D\mathbf{1}_{\|x-y\|\leq \underline{R}}\mathbf{1}_{\|y\|>M_n}\overline{\mu}^{N_k}(dy)\overline{\mu}_0^{S,N_k}(dx)
 \end{align*}

\begin{align}
       \mathbb{E}\left[\displaystyle\limsup_{k}\left(\displaystyle\frac{1}{N_k}\sum_{i\in\mathfrak{S}_{N_k}}\Sigma_2^{N_k}(X^i,n)\mathbf{1}_{\|X^i\|\leq M_n}\right)\right]&\leq \frac{Ce^{-aM_n^{\delta}}}{\left(\displaystyle \inf_{1\leq \ell \leq q(M_n+\underline{R})}\overline{\mu}(\mathfrak{A}_\ell)\right)^{\gamma}}\displaystyle \int_{\R^{2d}} \mathbf{1}_{\|x-y\|\leq \underline{R}}e^{-a\|x\|^{\delta}}dydx\nonumber\\
&\leq \frac{Ce^{-aM_n^{\delta}}}{\left(\displaystyle  \inf_{1\leq \ell\leq q(M_n+\underline{R});u\in\mathfrak{B}_\ell} c e^{-a\|u\|^{\delta}}\right)^{\gamma}}\ \displaystyle\left( \int_{B(0;\underline{R}} du \right)\int_{\R^{d}} e^{-a\|x\|^{\delta}}dx\nonumber\\
&\leq C e^{-a(1-\gamma(1+\epsilon))M_n^{\delta}} \nonumber\\
\label{lem7eq4}
 \mathbb{E}\left[\displaystyle\limsup_{k}\left(\displaystyle\frac{1}{N_k}\sum_{i\in\mathfrak{S}_{N_k}}\Sigma_2^{N_k}(X^i,n)\mathbf{1}_{\|X^i\|\leq M_n}\right)\right] &\leq C e^{-a(1-\gamma(1+\epsilon))M_n^{\delta}}.
\end{align}
\begin{align*}
\Sigma_3^{N_k}(X^i,n,s)\mathbf{1}_{\|X^i\|\leq M_n}&\leq \frac{C\mathbf{1}_{\|X^i\|\leq M_n}}{(\displaystyle \inf_{1\leq \ell \leq q (\|X^i\|+ \underline{R})}\overline{\mu}^{N_k}(\mathfrak{A}_\ell))^{\gamma}}\frac{1}{N_k}\sum_{j\in\mathfrak{S}_{N_k}}\left| A_{n,j}^{N_k}(s)-A_{j}^{N_k}(s)\right|\mathbf{1}_{\|X^j\|\leq M_n}\\
\displaystyle\frac{1}{N_k}\sum_{i\in\mathfrak{S}_{N_k}} \Sigma_3^{N_k}(X^i,n,s)\mathbf{1}_{\|X^i\|\leq M_n}&\leq C \displaystyle\frac{1}{N_k}\sum_{i\in\mathfrak{S}_{N_k}}\left(\displaystyle \inf_{1\leq \ell \leq q(\|X^i\|+ \underline{R})}\overline{\mu}^{N_k}(\mathfrak{A}_\ell)\right)^{-\gamma}\mathbf{1}_{\|X^i\|\leq M_n}\\
&\times \frac{1}{N_k}\sum_{j\in\mathfrak{S}_{N_k}}\left| A_{n,j}^{N_k}(s)-A_{j}^{N_k}(s)\right|\mathbf{1}_{\|X^j\|\leq M_n}\\
\displaystyle\limsup_{k}\left(\displaystyle\frac{1}{N_k}\sum_{i\in\mathfrak{S}_{N_k}}\Sigma_3^{N_k}(X^i,n,s)\mathbf{1}_{\|X^i\|\leq M_n}\right)&\leq C \limsup_{k} \left( \displaystyle \int_{B(0, M_n)} \left(\displaystyle \inf_{1\leq \ell \leq q (\|x\|+ \underline{R})}\overline{\mu}^{N_k}(\mathfrak{A}_\ell)\right)^{-\gamma}\overline{\mu}^{S,N_k}_0 (dx)\right)\\
&\times \limsup_{k} \left(\frac{1}{N_k}\sum_{j\in\mathfrak{S}_{N_k}}\sup_{v\in[0;s]}\left| A_{n,j}^{N_k}(s)-A_{j}^{N_k}(s)\right|\mathbf{1}_{\|X^j\|\leq M_n}\right).
\end{align*}
Exploiting Lemma \ref{lemEP}, we deduce that
\begin{align}
 \mathbb{E}&\left[\displaystyle\limsup_{k}\left(\displaystyle\frac{1}{N_k}\sum_{i\in\mathfrak{S}_{N_k}}\Sigma_3^{N_k}(X^i,n,s)\mathbf{1}_{\|X^i\|\leq M_n} \right)\right]\nonumber\\
 &\quad\leq  C\int_{B(0, M_n)}\left(\displaystyle \inf_{1\leq \ell \leq q(\|x\|+ \underline{R})}\overline{\mu}(\mathfrak{A}_\ell)\right)^{-\gamma}\overline{\mu}^{S}_0 (dx)\nonumber\\
 &\quad\times  \mathbb{E}\left[\displaystyle \limsup_{k}\left(\frac{1}{N_k}\sum_{j\in\mathfrak{S}_{N_k}}\sup_{v\in[0;s]}\left| A_{n,j}^{N_k}(v)-A_{j}^{N_k}(v)\right|\mathbf{1}_{\|X^j\|\leq M_n}\right)\right]\nonumber\\
  \label{lem7eq5}
  &\quad\leq C\mathbb{E}\left[\displaystyle \limsup_{k}\left(\frac{1}{N_k}\sum_{j\in\mathfrak{S}_{N_k}}\sup_{v\in[0;s]}\left| A_{n,j}^{N_k}(v)-A_{j}^{N_k}(v)\right|\mathbf{1}_{\|X^j\|\leq M_n}\right)\right]
\end{align}
From  \eqref{lem7eq3},  \eqref{lem7eq4} and \eqref{lem7eq5}, we have, with $\epsilon$ chosen such that $(1+\epsilon)\gamma<1$,
\begin{align}
\Psi^{N_k}_1(n,t)&\leq  C e^{-a(\frac{1}{1+\epsilon}-\gamma)M_n^{\delta}} t+ C e^{-a(1-\gamma(1+\epsilon))M_n^{\delta}}t\nonumber\\
\label{lem7eq8}
&+C \int_0^t \mathbb{E}\left[\displaystyle \limsup_{k}\left(\frac{1}{N_k}\sum_{i\in\mathfrak{S}_{N_k}}\sup_{v\in[0;s]}\left| A_{n,i}^{N_k}(v)-A_{i}^{N_k}(v)\right|\mathbf{1}_{\|X^i\|\leq M_n}\right)\right] ds
\end{align}
On the other hand, we apply Lemma \ref{lemEP} once again,
\begin{align*}
\overline{\Upsilon}_n^{i,N_k}(s,X^i )&=\left| A^{N_k}_{n,i}(s)-A^{N_k}_{i}(s) \right| \overline{\Gamma}_{n}^{N_k}(s,X^i)\\
&\leq  \frac{C\mathbf{1}_{\|X^j\|\leq M_n}}{\left(\displaystyle \inf_{1\leq \ell\leq q(\|X^j\|+r)}\overline{\mu}^{N_k}(\mathfrak{A}_\ell)\right)^{\gamma}}\left| A_{n,i}^{N_k}(s)-A_{i}^{N_k}(s)\right|\\
\overline{\Upsilon}^{i,N_k}(s,X^i,n )\mathbf{1}_{\|X^i\|\leq M_n}&\leq C\frac{1}{N_k}\sum_{j\in\mathfrak{S}_{N_k}}\displaystyle\left( \inf_{1\leq \ell \leq q(\|X^j\|+r)}\overline{\mu}^N(\mathfrak{A}_\ell)\right)^{-\gamma}\mathbf{1}_{\|X^j\|\leq M_n}\left| A_{n,i}^{N_k}(s)-A_{i}^{N_k}(s)\right|\mathbf{1}_{\|X^i\|\leq M_n} 
\end{align*}
\begin{align*}
\frac{1}{N_k}\sum_{i\in\mathfrak{S}_{N_k}}\sup_{v\in[0;s]} \overline{\Upsilon}^{i,N_k}(v,X^i,n )\mathbf{1}_{\|X^i\|\leq M_n}&\leq C\int_{B(0,M_n)} \displaystyle\left( \inf_{1\leq \ell \leq q(\|x\|+r)}\overline{\mu}^{N_k}(\mathfrak{A}_\ell)\right)^{-\gamma} \overline{\mu}^{S,N_k}_0 (dx)\\
&\times \frac{1}{N_k}\sum_{i\in\mathfrak{S}_{N_k}}\sup_{v\in[0;s]}\left| A_{n,i}^{N_k}(v)-A_{i}^{N_k}(v)\right|\mathbf{1}_{\|X^i\|\leq M_n} 
\end{align*}

$\displaystyle \mathbb{E}\left[\limsup_{k}\left(\frac{1}{N_k}\sum_{i\in\mathfrak{S}_{N_k}}\sup_{v\in[0;s]} \overline{\Upsilon}_n^{i,N_k}(v,X^i )\mathbf{1}_{\|X^i\|\leq M_n}\right)\right]\leq  C\int_{B(0,M_n)} \displaystyle\left( \inf_{1\leq \ell \leq q(\|x\|+r)}\overline{\mu}(\mathfrak{A}_\ell)\right)^{-\gamma} \overline{\mu}^{S}_0 (dx)\\
\displaystyle  \times \mathbb{E}\left[\displaystyle \limsup_{k}\left(\frac{1}{N_k}\sum_{i\in\mathfrak{S}_{N_k}}\sup_{v\in[0;s]}\left| A_{n,i}^{N_k}(v)-A_{i}^{N_k}(v)\right|\mathbf{1}_{\|X^i\|\leq M_n}\right)\right].$
 \begin{align}
  \label{lem7eq9}
\Psi^{N_k}_2(n,t))&\leq C\int_0^t \mathbb{E}\left[\displaystyle \limsup_{k}\left(\frac{1}{N_k}\sum_{i\in\mathfrak{S}_{N_k}}\sup_{v\in[0;s]}\left| A_{n,i}^{N_k}(v)-A_{i}^{N_k}(v)\right|\mathbf{1}_{\|X^i\|\leq M_n}\right)\right]ds 
\end{align} 
From \eqref{lem7eq1}, \eqref{lem7eq10}, \eqref{lem7eq8} and  \eqref{lem7eq9}, we have
\begin{align*}  
\mathbb{E}&\left[\displaystyle\limsup_{k}\left(\displaystyle\frac{1}{N_k}\sum_{i\in\mathfrak{S}_{N_k}}\sup_{s\in[0;t]}\left| A^{N_k}_{n,i}(s)-A^{N_k}_{i}(s)\right|\mathbf{1}_{\|X^i\|\leq M_n}\right)\right]\leq C e^{-a(\frac{1}{1+\epsilon}-\gamma)M_n^{\delta}} t+ C e^{-a(1-\gamma(1+\epsilon))M_n^{\delta}}t\\
 &\quad\quad\quad\quad\quad\quad+C\int_0^t \mathbb{E} \left[\displaystyle \limsup_{k}\left(\frac{1}{N_k}\sum_{i\in\mathfrak{S}_{N_k}}\sup_{v\in[0;s]}\left| A_{n,i}^{N_k}(v)-A_{i}^{N_k}(v)\right|\mathbf{1}_{\|X^i\|\leq M_n}\right)\right]ds\,.
 \end{align*}
 Using Gronwall's inequality, we obtain
 \begin{align}
  \mathbb{E}\left[\displaystyle\limsup_{k}\left(\displaystyle\frac{1}{N_k}\sum_{i\in\mathfrak{S}_{N_k}}\sup_{s\in[0;t]}\left| A^{N_k}_{n,i}(s)-A^{N_k}_{i}(s)\right|\mathbf{1}_{\|X^i\|\leq M_n}\right)\right]&\leq\left( C e^{-a(\frac{1}{1+\epsilon}-\gamma)M_n^{\delta}} t+ C e^{-a(1-\gamma(1+\epsilon))M_n^{\delta}}t\right)e^{Ct}\nonumber.
 \end{align}
 Thus, since $\epsilon$ has been chosen such that $(1+\epsilon)\gamma<1$,
\begin{align*}
\mathbb{E}\left[\displaystyle\limsup_{k}\left(\displaystyle\frac{1}{N_k}\sum_{i\in\mathfrak{S}_{N_k}}\sup_{s\in[0;t]}\left| A^{N_k}_{n,i}(s)-A^{N_k}_{i}(s)\right|\mathbf{1}_{\|X^i\|\leq M_n}\right)\right] \rightarrow 0,\quad \text{as} \quad n\rightarrow \infty.
\end{align*}
\end{proof}

\vfill
\eject

{\it Completing the proof of Theorem \ref{theo1}}

\smallskip

We  first  show that   $(\overline{\mu}^{S,N_k}, \overline{\mu}_t^{\mathfrak{F},N_k},\overline{\mu}_t^{I,N_k},\overline{\mu}_t^{R,N_k},  t\in \R_+)_{k \geq 1}$ converges in probability in $\mathfrak{D}^4_{\mathcal{M}},$ uniformly in $t$   to  $(\overline{\mu}_t^{S}, \overline{\mu}_t^{\mathfrak{F}},\overline{\mu}_t^{I},\overline{\mu}_t^{R}, t\in \R_+)$.\\
For all $t\in[0;T]$, we have, exploiting Lemma \ref{Lem imp2}
\begin{align}
\left|( \overline{\mu}_{t}^{S,N_k},\varphi )-( \overline{\mu}_{t}^{S},\varphi )\right|&\leq\left|( \overline{\mu}_{t}^{S,N_k}- \overline{\mu}_{t}^{S},\varphi \mathbf{1}_{B(0;M_n) })\right|+\left|( \overline{\mu}_{t}^{S,N_k}- \overline{\mu}_{t}^{S},\varphi \mathbf{1}_{B^c(0;M_n) })\right|\nonumber\\
&\leq \left|( \overline{\mu}_{t}^{S,N_k}- \overline{\mu}_{t}^{S},\varphi \mathbf{1}_{B(0;M_n) })\right|+\|\varphi\|_{\infty}\left|( \overline{\mu}^{N_k}, \mathbf{1}_{B^c(0;M_n) })\right|+\|\varphi\|_{\infty}\left|( \overline{\mu}, \mathbf{1}_{B^c(0;M_n) })\right|\nonumber\\
\label{pr1eq1}
\left|( \overline{\mu}_{t}^{S,N_k},\varphi )-( \overline{\mu}_{t}^{S},\varphi )\right|&\leq  \Lambda (t,N_k,n)+\|\varphi\|_{\infty}\left|( \overline{\mu}^{N_k}, \mathbf{1}_{B^c(0;M_n) })\right|+\|\varphi\|_{\infty}\left|( \overline{\mu}, \mathbf{1}_{B^c(0;M_n) })\right|
\end{align}
\begin{align}
  \Lambda (t,N_k,n) &\leq \left|( \overline{\mu}_{t}^{S,N_k}- \overline{\mu}_{n,t}^{S,N_k},\varphi \mathbf{1}_{B(0;M_n) })\right|+\left|(  \overline{\mu}_{n,t}^{S,N_k}- \overline{\mu}_{n,t}^{S},\varphi \mathbf{1}_{B(0;M_n) })\right|+\left|(  \overline{\mu}_{n,t}^{S}- \overline{\mu}_{t}^{S},\varphi \mathbf{1}_{B(0;M_n) })\right|\nonumber\\
  &\leq \|\varphi\|_{\infty}\left|( \overline{\mu}_{t}^{S,N_k}- \overline{\mu}_{n,t}^{S,N_k},\mathbf{1}_{B(0;M_n) })\right|+\left|(  \overline{\mu}_{n,t}^{S,N_k}- \overline{\mu}_{n,t}^{S},\varphi )\right|+\left|(  \overline{\mu}_{n,t}^{S}- \overline{\mu}_{t}^{S},\varphi )\right|\nonumber\\
  &\leq \frac{\|\varphi\|_{\infty}}{N_k}\sum_{i\in \mathfrak{S}_{N_k}}\mathbf{1}_{\|X^i\|>M_n}+\frac{\|\varphi\|_{\infty}}{N_k}\sum_{i\in\mathfrak{S}_{N_k}}\left|A_{n,i}^{N_k}(t)-A_{i}^{N_k}(t)\right|\mathbf{1}_{\|X^i\|\leq M_n}\nonumber\\
  \label{pr1eq2}
  &\quad+\left|(  \overline{\mu}_{n,t}^{S,N_k}- \overline{\mu}_{n,t}^{S},\varphi )\right|+\left|(  \overline{\mu}_{n,t}^{S}- \overline{\mu}_{t}^{S},\varphi )\right|
\end{align}
By combining \eqref{pr1eq1} and \eqref{pr1eq2}, we obtain
\begin{align*}
\left|( \overline{\mu}_{t}^{S,N_k},\varphi )-( \overline{\mu}_{t}^{S},\varphi )\right|&\leq \frac{\|\varphi\|_{\infty}}{N_k}\sum_{i\in \mathfrak{S}_{N_k}}\mathbf{1}_{\|X^i\|>M_n}+\frac{\|\varphi\|_{\infty}}{N_k}\sum_{i\in\mathfrak{S}_{N_k}}\left|A_{n,i}^{N_k}(t)-A_{i}^{N_k}(t)\right|\mathbf{1}_{\|X^i\|\leq M_n}\\
&\quad+\left|(  \overline{\mu}_{n,t}^{S,N_k}- \overline{\mu}_{n,t}^{S},\varphi )\right|+\left|(  \overline{\mu}_{n,t}^{S}- \overline{\mu}_{t}^{S},\varphi )\right|\\
&\quad+\|\varphi\|_{\infty}\left|( \overline{\mu}^{N_k}, \mathbf{1}_{B^c(0;M_n) })\right|+\|\varphi\|_{\infty}\left|( \overline{\mu}, \mathbf{1}_{B^c(0;M_n) })\right|
\end{align*}
From Lemma \ref{Lem imp2} and Theorem \ref{Thé1}, we deduce that
\begin{align}
\label{pr1eq3}
\mathbb{E}&\left[\limsup_{k}\left(\sup_{t\in[0;T]}\left|( \overline{\mu}_{t}^{S,N_k}-\overline{\mu}_{t}^{S},\varphi ) \right|\right)\right]\leq \sup_{t\in[0;T]}\left|(  \overline{\mu}_{n,t}^{S}- \overline{\mu}_{t}^{S},\varphi )\right|+3\|\varphi\|_{\infty}\left|( \overline{\mu}, \mathbf{1}_{B^c(0;M_n) })\right|\nonumber\\
&\quad+ \mathbb{E}\left[\displaystyle\limsup_{k}\left(\displaystyle\frac{1}{N_k}\sum_{i\in\mathfrak{S}_{N_k}}\sup_{t\in[0;T]}\left| A^{N_k}_{n,i}(s)-A^{N_k}_{i}(s)\right|\mathbf{1}_{\|X^i\|\leq M_n}\right)\right]\xrightarrow[n \to \infty]{} 0
\end{align}
We next consider the other measures.
\begin{align*}
\left|( \overline{\mu}_{t}^{I,N_k},\varphi )-( \overline{\mu}_{t}^{I},\varphi )\right|&\leq \frac{\|\varphi\|_{\infty}}{N_k}\sum_{i\in \mathfrak{I}_{N_k}}\mathbf{1}_{\|X^i\|>M_n}+\frac{\|\varphi\|_{\infty}}{N_k}\sum_{i\in\mathfrak{S}_{N_k}}\left|A_{n,i}^{N_k}(t)-A_{i}^{N_k}(t)\right|\mathbf{1}_{\|X^i\|\leq M_n}\\
&\quad+\left|(  \overline{\mu}_{n,t}^{I,N_k}- \overline{\mu}_{n,t}^{I},\varphi \mathbf{1}_{B(0;M_n)})\right|
  +\left|(  \overline{\mu}_{n,t}^{I}- \overline{\mu}_{t}^{I},\varphi \mathbf{1}_{B(0;M_n)})\right|\\
  &\quad+\|\varphi\|_{\infty}\left|( \overline{\mu}^{N_k}, \mathbf{1}_{B^c(0;M_n) })\right|+\|\varphi\|_{\infty}\left|( \overline{\mu}, \mathbf{1}_{B^c(0;M_n) })\right|\\
  \left|( \overline{\mu}_{t}^{R,N_k},\varphi )-( \overline{\mu}_{t}^{R},\varphi )\right|&\leq \frac{\|\varphi\|_{\infty}}{N_k}\sum_{i\in \mathfrak{R}_{N_k}}\mathbf{1}_{\|X^i\|>M_n}+\frac{\|\varphi\|_{\infty}}{N_k}\sum_{i\in \mathfrak{I}_{N_k}}\mathbf{1}_{\|X^i\|>M_n}\\
  &\quad+\frac{\|\varphi\|_{\infty}}{N_k}\sum_{i\in\mathfrak{S}_N}\left|A_{n,i}^{N_k}(t)-A_{i}^{N_k}(t)\right|\mathbf{1}_{\|X^i\|\leq M_n}\nonumber\\
  &\quad+\left|(  \overline{\mu}_{n,t}^{R,N_k}- \overline{\mu}_{n,t}^{R},\varphi  \mathbf{1}_{B(0;M_n)})\right|
  +\left|(  \overline{\mu}_{n,t}^{R}- \overline{\mu}_{t}^{R},\varphi \mathbf{1}_{B(0;M_n)})\right|\\
  &\quad+\|\varphi\|_{\infty}\left|( \overline{\mu}^{N_k}, \mathbf{1}_{B^c(0;M_n)})\right|
  +\|\varphi\|_{\infty}\left|( \overline{\mu}, \mathbf{1}_{B^c(0;M_n) })\right|\\
  \left|( \overline{\mu}_{t}^{\mathfrak{F},N_k}-\overline{\mu}_{t}^{\mathfrak{F}},\varphi ) \right|&\leq\frac{\lambda^*\|\varphi\|_{\infty}}{N_k}\sum_{i\in \mathfrak{I}_{N_k}}\mathbf{1}_{\|X^i\|>M_n}+\frac{\lambda^*\|\varphi\|_{\infty}}{N_k}\sum_{i\in\mathfrak{S}_{N_k}}\left|A_{n,i}^{N_k}(t)-A_{i}^{N_k}(t)\right|\mathbf{1}_{\|X^i\|\leq M_n}\\
  &\quad+\left|(  \overline{\mu}_{n,t}^{\mathfrak{F},N_k}- \overline{\mu}_{n,t}^{\mathfrak{F}},\varphi \mathbf{1}_{B(0;M_n)})\right|
  +\left|(  \overline{\mu}_{n,t}^{\mathfrak{F}}- \overline{\mu}_{t}^{\mathfrak{F}},\varphi \mathbf{1}_{B(0;M_n)})\right|\\
  &\quad+\lambda^* \|\varphi\|_{\infty}\left|( \overline{\mu}^{N_k}, \mathbf{1}_{B^c(0;M_n) })\right|+\lambda^*\ \|\varphi\|_{\infty}\left|( \overline{\mu}, \mathbf{1}_{B^c(0;M_n) })\right|
\end{align*}

Thus \eqref{pr1eq3} holds with $\overline{\mu}_{t}^{S,N_k}-\overline{\mu}_{t}^{S}$ replaced by 
$\overline{\mu}_{t}^{\mathfrak{F},N_k}-\overline{\mu}_{t}^{\mathfrak{F}}$, $\overline{\mu}_{t}^{I,N_k}-\overline{\mu}_{t}^{I}$ and $\overline{\mu}_{t}^{R,N_k}-\overline{\mu}_{t}^{R}$, and all those quantities tend to $0$ in probability in $\mathfrak{D}_{\mathcal{M}}$.

The above arguments show that from any subsequence of the original sequence, we can further extract a subsequence $\{N_k,\ k\ge1\}$ such that  $( \overline{\mu}^{S,N_k}, \overline{\mu}^{\mathfrak{F},N_k},  \overline{\mu}^{I,N_k},  \overline{\mu}^{R,N_k} )$ converges in probability in $\mathfrak{D}_{\mathcal{M}}^4$ to $(  \overline{\mu}^S,  \overline{\mu}^{\mathfrak{F}} ,  \overline{\mu}^I,  \overline{\mu}^R)$. This clearly implies that the whole sequence converges, hence Theorem \ref{theo1}  is established.


\begin{thebibliography}{99}

\bibitem{LJ-BM}
L.J.S. Allen, B.M. Bolker, Y.~Lou, and A.L. Nevai
\newblock  Asymptotic profiles of the steady states for an sis epidemic
  reaction diffusion model.
\newblock {\em Discrete Contin. Dyn. Syst.}, pages 1--20, 2008.

\bibitem{HA-TB}
H.~Anderson and T.~Britton.
\newblock {\em Stochastic epidemic model and their statistical analysis}.
\newblock Lecture Notes in Statistics. vol 151. Springer Science and Business
  Media, 2012.

\bibitem{SB-AE-EP}
S.~Bowong, A.~Emakoua, and E.~Pardoux.
\newblock A spatial stochastic epidemic model: law of large numbers and central
  limit.
\newblock {\em Stoch. Partial Differ.Equ.Anal.Comput.}, 11(1):31--105, 2023.

\bibitem{TB-EP}
T.~Britton and E.~Pardoux.
\newblock {\em Stochastic epidemic models with inference (eds)}.
\newblock Lecture Notes in Math. vol 2255. Springer, cham, 2019.

\bibitem{FPP}
R.~Forien, G.~Pang, and E.~Pardoux.
\newblock Epidemic models with varying infectivity.
\newblock {\em SIAM J. Appl. Math}, 81:1893--1930, 2021.

\bibitem{RF-GP-EP}
R.~Forien, G.~Pang, and E.~Pardoux.
\newblock Multi-patch multi-group model with varying infectivity.
\newblock {\em Probab. Uncertain. Quant. Risk}, 7(4):333--364, 2022.

\bibitem{Kaj}
I.~Kaj.
\newblock A weak interaction epidemic among diffusive particles.
\newblock {\em Stochastic Partial Differential Equations}, page 189–208,
  1995.

\bibitem{WO-AG}
W.O Kermack and A.G McKendrick.
\newblock A contribution to the mathematical theory of epidemic.
\newblock {\em Proc. of the Royal Soc. A}, 115:700--721, 1927.


\bibitem{AK-EP}
A.~Kanga, E.~Pardoux.
\newblock Spatially dense stochastic epidemic models with infection-age
  dependent infectivity.Spatial SIR epidemic model with varying infectivity without movement of individuals: Law of large numbers.
  \newblock {\em Pure and Applied Funct. Analysis}, to appear.

\bibitem{MN-EP-TY}
M.~N’zi, E.~Pardoux, and T.~Yeo.
\newblock A sir model on a refining spatial grid i - law of large numbers.
\newblock {\em Applied Math and Optimization}, 83:1153--1189, 2021.

\bibitem{PPdense}
G.~Pang and {\'E}.~Pardoux.
\newblock Spatially dense stochastic epidemic models with infection-age
  dependent infectivity.
\newblock {\em arXiv preprint arXiv:2304.05211}, 2023.

\bibitem{rao}
R.~R. Rao.
\newblock The law of large numbers for d[0,1]-valued random variables.
\newblock {\em Theory of Probability \& Its Applications}, 8:75--79, 1963.

\bibitem{Rienert}
G.~Reinert.
\newblock The asymptotic evolution of the general stochastic epidemic.
\newblock {\em The Annals of Applied Probability}, 5:1061--1086, 1995.

\bibitem{szni}
A.~S. Sznitman.
\newblock Topics in propagation of chaos.
\newblock In P.L. Hennequin, editor, {\em Ecole de Probabilit\'es de
  Saint-Flour XIX}, volume 1464 of {\em Lecture Notes in Mathematics}, pages
  165--251. Springer, 1989.

\bibitem{YV-MH-EP}
Y.~V. Vuong, M.~Haurray, and E.~Pardoux.
\newblock Conditional propagation of choas in a spatial stochastic epidemic
  model with common noise.
\newblock {\em Stoch. Partial Differ. Equ. Anal. Comput.}, 10(3):1180--1210,
  2022.

\bibitem{Wang}
F.~J.~S. Wang.
\newblock Limit theorems for age and density dependent stochastic population
  models.
\newblock {\em Journal of Mathematical Biology}, 2:373--400, 1975.


\end{thebibliography}
\end{document}